\RequirePackage{etoolbox}
\csdef{input@path}{{style/}{graphics/}}
\documentclass[ba,linksfromyear,nameyear,dvips,preprint]{imsart}
\usepackage{natbib}
\usepackage[backref=page]{hyperref}
\usepackage{amsthm,amsmath}
\usepackage{amssymb}
\usepackage{graphicx}
\usepackage{vtexurl}

\usepackage{bm}
\usepackage{graphicx}

\pubyear{2015}
\volume{10}
\issue{1}
\firstpage{189}
\lastpage{221}
\doi{10.1214/14-BA915}

\def\bfI{\bm{I}}
\def\bfR{\bm{R}}
\def\bfX{\bm{X}}
\def\bfa{\bm{a}}
\def\bfp{\bm{p}}
\def\bft{\bm{t}}
\def\bfz{\bm{z}}
\def\bfmu{\bm{\mu}}
\def\bfGamma{\bm{\Gamma}}
\def\bfTheta{\bm{\Theta}}
\def\bfSigma{\bm{\Sigma}}
\def\bfOmega{\bm{\Omega}}
\def\bfPsi{\bm{\Psi}}

\newcommand{\cfA}{\mathcal{A}}
\newcommand{\cfC}{\mathcal{C}}
\newcommand{\cfF}{\mathcal{F}}
\newcommand{\cfN}{\mathcal{N}}
\newcommand{\cfP}{\mathcal{P}}
\newcommand{\cfR}{\mathcal{R}}
\newcommand{\cfT}{\mathcal{T}}
\newcommand{\cfX}{\mathcal{X}}

\newcommand{\g}{\,\vert\,}
\newcommand{\E}{\mbox{E}}
\newcommand{\Var}{\mbox{Var}}

\newcommand{\Be}{\mbox{Be}}
\newcommand{\Bi}{\mbox{Bi}}
\newcommand{\Di}{\mbox{Di}}
\newcommand{\Ga}{\mbox{Ga}}
\newcommand{\No}{\mbox{N}}
\newcommand{\St}{\mbox{St}}

\newcommand{\comb}[2]{{#1\choose #2}}
\newcommand{\h}{\hbox{{\small$1\over2$}}}
\newcommand{\m}[1]{\;\raise2pt\hbox{#1}}
\newcommand{\w}[1]{\mbox{#1}}

\newcommand{\wileysize}{
\setlength{\parindent}{10pt}
\setlength{\parskip}{2pt plus1pt}}
\newtheorem{teo}{Proposition}[section]
\newtheorem{ex}{Example}[section]
\newenvironment{exa}{\begin{list}{}{\vspace*{-24pt}\setlength{\leftmargin}{10pt}\setlength{\rightmargin}{\leftmargin}}\item\begin{ex}\em}{\end{ex}\end{list}}
\wileysize
\newcommand{\go}{\rightarrow}
\newcommand{\fr}{\frac}
\newtheorem{theorem0}{Theorem}[section]
\newtheorem{lemma0}{Lemma}[section]
\newtheorem{remark0}{Remark}[section]
\newtheorem{fact0}{Fact}[section]
\newtheorem{example0}{Example}[section]
\newtheorem{definition0}{Definition}[section]
\newtheorem{corollary0}{Corollary}[section]
\newtheorem{proposition0}{Proposition}[section]
\newtheorem{algorithmY}{Algorithm}[section]
\newenvironment{theorem}{\begin{theorem0} \mbox{} }{\end{theorem0}}

\newcommand{\reals}{\mbox{\textrm{I\kern-.20em R}}}
\newcommand{\expect}{\E}
\newtheorem{dfn}{Definition}
\newcommand{\bdfn}{\begin{dfn}\itshape}
\newcommand{\bdfnn}[1]{\begin{dfn}{\mbf(#1).}\itshape}
\newcommand{\edfn}{\end{dfn}}
\newcommand{\edfnn}{\end{dfn}}

\endlocaldefs

\begin{document}

\begin{frontmatter}

\vspace*{5pt}
\title{Overall Objective Priors\thanksref{T1}}
\runtitle{Overall Objective Priors}

\relateddois{T1}{Related articles:
DOI:~\relateddoi[ms=BA935,title={Siva Sivaganesan. Comment on Article by Berger, Bernardo, and~Sun}]{Related item:}{10.1214/14-BA935},
DOI:~\relateddoi[ms=BA936,title={Manuel Mendoza, Eduardo Guti\'errez-Pe\~na. Comment on Article by Berger, Bernardo, and~Sun}]{Related item:}{10.1214/14-BA936},
DOI:~\relateddoi[ms=BA937,title={Judith Rousseau. Comment on Article by Berger, Bernardo, and Sun}]{Related item:}{10.1214/14-BA937},\hfill\break
DOI:~\relateddoi[ms=BA938,title={Gauri Sankar Datta, Brunero Liseo. Comment on Article by Berger, Bernardo, and Sun}]{Related item:}{10.1214/14-BA938};
rejoinder at
DOI:~\relateddoi[ms=BA943,title={James O. Berger, Jose M. Bernardo, Dongchu Sun. Rejoinder}]{Related item:}{10.1214/15-BA943}.}

\begin{aug}
\author{\fnms{James O.} \snm{Berger}\thanksref{t1}},
\author{\fnms{Jose M.} \snm{Bernardo}\thanksref{t2}},
\and
\author{\fnms{Dongchu} \snm{Sun}\thanksref{t3}}

\runauthor{J. O. Berger, J. M. Bernardo, and D. Sun}


\thankstext{t1}{Duke University, USA and King Abdulaziz University,
Saudi Arabia,
{berger@stat.duke.edu}}
\thankstext{t2}{Universitat de Val{\`e}ncia, Spain,
{jose.m.bernardo@uv.es}}
\thankstext{t3}{University of Missouri-Columbia, USA,
{sund@missouri.edu}}

\end{aug}

%
\begin{abstract}
In multi-parameter models, reference priors typically
depend on the parameter or quantity of interest, and
it is well known that this is necessary to produce objective posterior
distributions
with optimal properties. There are, however, many situations
where one is simultaneously interested in all the parameters
of the model or, more realistically, in functions of them that include
aspects such as prediction,
and it would then be useful to have a single objective prior
that could safely be used to produce reasonable posterior
inferences for all the quantities of interest. In this paper,
we consider three methods for selecting a single objective prior
and study, in a variety of problems including the multinomial problem,
whether or not the resulting prior is a reasonable overall prior.
\end{abstract}

%


\end{frontmatter}


\section{Introduction}

\subsection{The problem}

Objective Bayesian methods, where the formal prior distribution is
derived from the assumed model rather than assessed from expert
opinions, have a long history (see \emph{e.g.}, Bernardo and Smith,
\citeyear
{BerSmi1994};
\citealp{KasWas1996}, and references therein). Reference priors
(Bernardo, \citeyear{Ber1979}, \citeyear{Ber2005};
Berger and Bernardo, \citeyear{BerBer1989}, \citeyear{BerBer1992a},b,
Berger, Bernardo and Sun, \citeyear{BerBerSun2009}, \citeyear{BerBerSun2012})
are a popular choice of objective prior.
Other interesting developments involving objective priors include
\citet{ClaBar1994}, \citet{ClaYua2004},
\citet{ConVerGut2004},
\citet{De2001},
\citet{De2006},
Datta and Ghosh (\citeyear{DatGho1995a,DatGho1995b}),
\citet{DatGho1996},
\citet{Dat2000},
\citet{Gho2011},
\citet{GhoMerLiu2011},
\citet{GhoRam2003},
\citet{Lis1993},
\citet{LisLop2006},
\citet{Siv1994},
\citet{SivLauMue2011} and
\citet{WalGut2011}.

In single parameter problems, the reference prior is uniquely defined
and is invariant under reparameterization. However, in multiparameter
models, the reference prior depends on the quantity of interest, \emph{e.g.},
the parameter concerning which inference is being performed.
Thus, if data $\bm{x}$ are assumed to have been generated from $p(\bm{x}
\g
\bm{\omega})$, with
$\bm{\omega}\in\bfOmega\subset\Re^k$, and one is interested in
$\theta
(\bm{\omega})\in\Theta\subset\Re$,
the reference prior $\pi_{\theta}(\bm{\omega})$, will typically depend
on~$\theta$;
the posterior distribution, $\pi_{\theta}(\bm{\omega}\g\bm
{x})\propto
p(\bm{x}
\g\bm{\omega})\,\pi_{\theta}(\bm{\omega})$,
thus also depends on $\theta$, and inference for $\theta$ is
performed using
the corresponding marginal reference posterior for $\theta(\bm{\omega})$,
denoted $\pi_{\theta}(\theta\g\bm{x})$.
The dependence of the reference prior on the quantity of interest has
proved necessary to obtain objective posteriors with appropriate
properties -- in particular, to have good frequentist coverage
properties (when attainable) and to avoid marginalization paradoxes and
strong inconsistencies.

There are however many situations where one is {\em simultaneously}
interested in all the parameters of the model or perhaps in several
functions of them. Also, in prediction and decision analysis,
parameters are not themselves the object of direct interest and yet an
overall prior is needed to carry out the analysis. Another situation in
which having an overall prior would be beneficial is when a user is
interested in a non-standard quantity of interest (\emph{e.g.}, a non-standard
function of the model parameters), and is not willing or able to
formally derive the reference prior for this quantity of interest.
Computation can also be a consideration; having to separately do
Bayesian computations with a different reference prior for each
parameter can be onerous. Finally, when dealing with non-specialists it
may be best pedagogically to just present them with one overall
objective prior, rather than attempting to explain the technical
reasons for preferring different reference priors for different
quantities of interest.

To proceed, let
$\bm{\theta}= \bm{\theta}(\bm{\omega})=\{\theta_1(\bm{\omega
}),\ldots,\theta
_m(\bm{\omega})\}$
be the set of $m>1$ functions of interest. Our goal is to find a joint prior
$\pi(\bm{\omega})$
whose corresponding marginal posteriors, $\{\pi(\theta_i\g\bm{x})\}
_{i=1}^m$, are sensible from a reference
prior perspective. This is not a well-defined goal, and so we will
explore various possible approaches to the problem.

\begin{exa}
{\bf Multinomial Example:} Suppose ${\bm{x}} =(x_1, \ldots, x_m)$ is
multinomial $\mbox{Mu}(\bm{x}\g n; \theta_1, \ldots, \theta_m)$, where
$\sum_{i=1}^m x_i=n$, and $\sum_{i=1}^m \theta_i=1$. In
\citet{BerBer1992b}, the reference prior is derived when the parameter
$\theta_{i}$ is
of interest, and this is a different prior for each $\theta_{i}$, as
given in the paper. The reference prior for $\theta_{i}$ results in a
Beta reference marginal posterior
$\Be(\theta_{i}\g x_{i}+\h, n-x_{i}+\h)$.
We would like to identify a single joint prior for $\bm{\theta}$ whose
marginal posteriors could be expected to be close to each of these
reference marginal posteriors, in some average sense.
\end{exa}

\subsection{Background}

It is useful to begin by recalling earlier efforts at obtaining an
overall reference prior.
There have certainly been analyses that can be interpreted as informal
efforts at obtaining an
overall reference prior. One example is given in \citet{BerSun2008}
for the five parameter bivariate normal model.
Priors for all the quantities of interest that had previously been
considered for the bivariate normal model
(21 in all) were studied from a
variety of perspectives. One such perspective was that of finding a
good overall prior, defined as one which yielded
reasonable frequentist coverage properties when used for at least the
most important quantities of interest.
The conclusion was that the prior $\pi^o(\mu_1,\mu_2,\sigma
_1,\sigma
_2,\rho) = 1/[\sigma_1 \sigma_2 (1-\rho^2)]$,
where the $\mu_i$ are the means, the $\sigma_i$ are the standard
deviations, and $\rho$ is the correlation
in the bivariate normal model,
was a good choice for the overall prior.

We now turn to some of the more formal efforts to create an overall
objective prior.

\subsubsection{Invariance-based priors}

If $p(\bm{x}\g\bm{\omega})$ has a group invariance structure, then the
recommended objective prior is \mbox{typically} the right-Haar prior.
Often this will work well for all\vadjust{\eject} parameters that
define the
\mbox{invariance} structure. For instance, if the sampling model is
$\No( x_i \g\mu, \sigma)$, the right-Haar prior is \mbox{$\pi(\mu
,\sigma) =\sigma^{-1}$,}
and this is fine for either $\mu$ or $\sigma$ (yielding the usual
objective posteriors). Such a nice situation
does not always obtain, however.

\begin{exa}
{\bf Bivariate Normal Distribution:} The right-Haar prior is not unique
for the bivariate normal problem. For instance, two
possible right-Haar priors are $\pi_1(\mu_1,\mu_2,\sigma_1,\sigma
_2,\rho
) = 1/[\sigma_1^2 (1-\rho^2)]$
and $\pi_2(\mu_1,\mu_2,\sigma_1,\sigma_2,\rho) = 1/[\sigma_2^2
(1-\rho
^2)]$. In \citet{BerSun2008}
it is shown that $\pi_i$ is fine for $\mu_i$, $\sigma_i$ and $\rho$,
but leads to problematical
posteriors for the other mean and standard deviation.

\end{exa}

The situation can be even worse if the right-Haar prior is used for
other parameters that can be considered.

\begin{exa}
{\bf Multi-Normal Means:}
\label{ex.multi-normal} Let $x_i$ be independent normal with mean $\mu
_i$ and variance $1$, for $i=1,\cdots,m$.
The right-Haar prior for $\bfmu=(\mu_1, \ldots, \mu_m)$ is just
a constant, which is fine for each of the individual normal means,
resulting in a sensible $\No(\mu_i \g x_i, 1)$ posterior
for each individual $\mu_i$.
But this prior is bad for overall quantities such as $\theta= \frac
{1}{m}|\bfmu|^2=\frac{1}{m}\sum_{i=1}^m \mu_i^2$, as discussed in
\citet{Ste1959} and Bernardo and Smith (1994, p.~365).
For instance, the resulting posterior mean of $\theta$ is $[1+\frac
{1}{m}\sum_{i=1}^m x_i^2 ]$, which is inconsistent as
$m \rightarrow\infty$ (assuming $\frac{1}{m}\sum_{i=1}^m \mu_i^2$ has
a limit); indeed, it is easy to show that then $[1+\frac{1}{m}\sum
_{i=1}^m x_i^2 ] \rightarrow[\theta_T +2]$, where $\theta_T$ is the
true value of $\theta$. Furthermore, the posterior distribution of
$\theta$ concentrates sharply around this
incorrect value.
\end{exa}

\subsubsection{Constant and vague proper priors}

\citet{Lap1812} advocated use of a constant prior as the overall
objective prior and this approach, eventually named {\em inverse probability},
dominated statistical practice for over~100 years. But the problems of
a constant prior are well-documented, including the following:
\begin{enumerate}
\item Lack of invariance to transformation, the main criticism directed
at Laplace's approach.
\item Frequent posterior impropriety.
\item Possible terrible performance, as in the earlier multi-normal
mean example.
\end{enumerate}

Vague proper priors (such as a constant prior over a large compact set)
are perceived by many as being adequate as an overall objective prior,
but they too have well-understood problems. Indeed, they are, at best,
equivalent to use of a constant prior, and so inherit most of the flaws
of a constant prior. In the multi-normal mean example, for instance,
use of $\No(\mu_i\g0, 1000)$ vague proper
priors results in a posterior mean for $\theta$ that is virtually
identical to the inconsistent posterior mean from the constant prior.

There is a common misperception that vague proper priors are safer than
a constant prior, since a proper posterior is guaranteed with
a vague proper prior but not for a constant prior.\vadjust{\eject}
But this actually
makes vague proper priors more dangerous than a constant prior.
When the constant prior results in an improper posterior distribution,
the vague proper prior will yield an essentially arbitrary posterior,
depending on the degree of vagueness that is chosen for the prior. And
to detect that the answer is arbitrary, one has to conduct a sensitivity
study concerning the degree of vagueness, something that can be
difficult in complex problems when several or high-dimensional vague
proper priors are used.
With the constant prior on the
other hand, the impropriety of the posterior will usually show up in
the computation---the Markov Chain Monte Carlo (MCMC) will not
converge---and hence can be recognized.

\subsubsection{Jeffreys-rule prior}

The Jeffreys-rule prior (Jeffreys, \citeyear{Jef1946}, \citeyear
{Jef1961}) is the same for all
parameters in a model, and is, hence, an obvious candidate for an
overall prior.
If the data
model density is $p(\bm{x}\g\bm{\theta})$ the Jeffeys-rule prior for
the unknown $\bm{\theta}=\{\theta_1,\ldots,\theta_m\}$ has the form
\[
\pi(\theta_1,\ldots,\theta_m)=|I(\bm{\theta})|^{1/2} ,
\]
where $I(\bm{\theta})$
is the $m\times m$ Fisher information matrix with $(i,j)$
element
\[
I(\bm{\theta})_{ij} =\E_{\bm{x}\g\bm{\theta}}\bigg[-
\frac{\partial^2}{\partial\theta_i\partial\theta_j} \log
p(\bm{x}\g\bm{\theta})\bigg]\,.
\]

This is the optimal objective prior (from many perspectives) for {\em
regular one-parameter} models, but has problems
for multi-parameter models. For instance, the right-Haar
prior in the earlier multi-normal mean problem is also the
Jeffreys-rule prior there, and was seen to result in an inconsistent estimator
of $\theta$. Even
for the basic $\No( x_i \g\mu, \sigma)$ model, the Jeffreys-rule prior
is $\pi(\mu,\sigma) = 1/ \sigma^2$, which results
in posterior inferences for $\mu$ and~$\sigma$ that have the wrong
`degrees of freedom.'

For the bivariate normal example, the Jeffreys-rule prior is $1/[\sigma
_1^2 \sigma_2^2 (1-\rho^2)^2]$; this yields the
natural marginal posteriors for the means and standard deviations, but
results in quite inferior objective posteriors for
$\rho$ and various derived parameters, as shown in \citet{BerSun2008}.
More, generally, the Jeffreys-rule prior for
a covariance matrix is studied in \citet{1994}, and shown to
yield a decidedly inferior posterior.

There have been efforts to improve upon the Jeffreys-rule prior, such
as consideration of the ``independence Jeffreys-rule prior," but
a general alternative definition has not resulted.

Finally, consider the following well-known example, which suggests
problems with the Jeffreys-rule prior even when it is proper.

\begin{exa}
\label{ex.multinomialspec}
{\bf Multinomial Distribution (continued):} Consider the multinomial
\mbox{example} where the sample size $n$ is small relative to
the number of classes~$m$; thus we have a large sparse table. The
Jeffreys-rule prior is the {\em proper} prior,
$\pi(\theta_1, \ldots, \theta_m) \propto\prod_{i=1}^{m} \theta
_i^{-1/2}$ , but is not a good candidate for the overall prior. For
instance, suppose $n=3$ and $m=1000$, with $x_{240}=2$, $x_{876}=1$,
and all the other $x_i=0$. The posterior means resulting from use of
the Jeffreys-rule prior are
\[
\E[\theta_i \g{\bm{x}}] = \frac{x_i+1/2}{\sum_{i=1}^{m}
(x_i+1/2)} =
\frac{x_i+1/2}{n + m/2}= \frac{x_i+1/2}{503} \m,
\]
so $\E[\theta_{240} \g{\bm{x}}] = \frac{2.5}{503}$, $\E[\theta_{876}
\g{\bm{x}}] = \frac{1.5}{503}$, $\E[\theta_{i} \g{\bm{x}}] =
\frac
{0.5}{503}$ otherwise.
So, cells 240 and 876 only have total posterior probability of $\frac
{4}{503} = 0.008$ even though all 3 observations are in these cells.
The problem is that the Jeffreys-rule prior effectively added 1/2 to
the 998 zero cells, making them more important than the cells with
data! That the Jeffreys-rule prior can encode much more information
than is contained in the data is hardly desirable
for an objective analysis.

An alternative overall prior that is sometimes considered is the
uniform prior on the simplex, but
this is even worse than the Jeffreys prior, adding 1 to each cell. The
prior that adds
0 to each cell is $\prod_{i} \theta_i^{-1}$, but this results in an
improper posterior if any cell has a zero entry, a virtual
certainty for very large tables.
\end{exa}

We actually know of no multivariable example in which we would
recommend the Jeffreys-rule prior. In higher dimensions, the prior
always seems to be either `too diffuse' as in the multinormal means
example, or `too concentrated' as in the multinomial example.

\subsubsection{Prior averaging approach}
Starting with a collection of reference (or other) priors
$\{\pi_i(\bm{\theta}), i=1,\ldots,m\}$ for differing parameters or
quantities of interest,
a rather natural approach is to use an average of the priors. Two
natural averages to
consider are the arithmetic mean
\[
\pi^A(\bm{\theta}) = \frac{1}{m}\sum\nolimits_{i=1}^m \pi_i(\bm
{\theta})
\, ,
\]
and the geometric mean
\[
\pi^G(\bm{\theta})
=\prod\nolimits_{i=1}^m \pi_i(\bm{\theta})^{1/m}\, .
\]
While the arithmetic average might seem most natural, arising from the
hierarchical reasoning of
assigning each $\pi_i$ probability $1/m$ of being correct, geometric
averaging arises naturally
in the definition of reference priors (\citealp{BerBerSun2009}),
and also is the optimal prior if one is trying to choose a single prior
to minimize the average of
the Kullback-Leibler (KL) divergences of the prior from the $\pi_i$'s
(a fact of which we were
reminded by Gauri Datta). Furthermore, the weights in arithmetic
averaging of improper priors are rather arbitrary
because the priors have no normalizing constants, whereas geometric
averaging is unaffected by
normalizing constants.

\begin{exa}
{\bf Bivariate Normal Distribution (continued):}
Faced with the two right-Haar priors in this problem,
\[
\pi_1(\mu_1,\mu_2,\sigma_1,\sigma_2,\rho) =\sigma_1^{-2} (1-\rho
^2)^{-1},\qquad
\pi_2(\mu_1,\mu_2,\sigma_1,\sigma_2,\rho) =\sigma_2^{-2} (1-\rho
^2)^{-1},
\]
the two average priors are
\begin{eqnarray}
\label{eq.S}
\pi^A(\mu_1,\mu_2,\sigma_1,\sigma_2,\rho) &=& \frac{1}{2\sigma_1^2
(1-\rho^2)} + \frac{1}{2\sigma_2^2 (1-\rho^2)}\m,\\
\pi^G(\mu_1,\mu_2,\sigma_1,\sigma_2,\rho) &=& \frac{1}{\sigma_1
\sigma
_2 (1-\rho^2)} \,.
\end{eqnarray}
\vadjust{\eject}

Interestingly, \citet{SunBer2007} show that $\pi^A$ is a worse
objective prior than either right-Haar prior alone,
while $\pi^G$ is the overall recommended objective prior.
\end{exa}

One problem with the averaging approach is that each of the reference priors
can depend on all of the other parameters, and not just the parameter
of interest, $\theta_i$, for which it was created.
\begin{exa}
{\bf Multinomial Example (continued):} The reference prior derived when
the parameter of interest is~$\theta_i$ actually depends on the
sequential ordering chosen for
all the parameters (e.g. $\{\theta_i, \theta_1, \theta_2, \ldots,
\theta
_{i-1}, \theta_{i+1},\ldots,\theta_m\}$) in the reference prior derivation;
there are thus $(m-1)!$ different reference priors for each parameter
of interest.
Each of these reference priors will result in the same marginal
reference posterior
for~$\theta_i$,
\[
\pi_{\theta_i}(\theta_i\g\bm{x})=\Be(\theta_{i}\g x_{i}+\h,
n-x_{i}+\h),
\]
but the full reference prior and the full posterior,
$\pi_{\theta_i}(\theta\g\bm{x})$, do depend on the ordering of the
other parameters. There are
thus a total of $m!$ such full reference priors to be averaged, leading
to an often-prohibitive computation.
\end{exa}

In general, the quality of reference priors as overall priors is
unclear, so there is no obvious sense in which an
average of them will make a good overall reference prior. The prior
averaging approach is thus best viewed
as a method of generating interesting possible priors for further
study, and so will not be considered further herein.

\subsection{Three approaches to construction of the overall prior}

\subsubsection{Common reference prior}
If the reference prior that is computed for any parameter of the model
(when declared to be the parameter of interest)
is the same, then this common reference prior is the natural choice for
the overall prior. This is illustrated
extensively in Section \ref{sec.common}; indeed, the section attempts
to catalogue the situations in which this
is known to be the case, so that these are the situations with a
ready-made overall prior.

\subsubsection{Reference distance approach}
In this approach, one seeks a prior that will yield marginal
posteriors, for each $\theta_i$ of interest, that
are close to the set of reference posteriors
$\{\pi(\theta_i\g\bm{x})\}_{i=1}^m$ (yielded by the set of
reference priors
$\{\pi_{\theta_i}(\bm{\omega})\}_{i=1}^m$), in an average sense
over both
posteriors and data $\bm{x}\in\cfX$.

\begin{exa}
{\bf Multinomial Example (continued):} In Example~\ref
{ex.multinomialspec} consider, as an overall prior, the Dirichlet $\Di
(\bm{\theta}\g a, \ldots, a)$ distribution, having density
proportional to
$\prod_i \theta_i^{a-1}$, leading to
$\Be(\theta_{i}\g x_{i}+a, n-x_{i}+(m-1)a)$ as the marginal posterior
for $\theta_i$. In Section~\ref{sub.multinomial}, we will study which
choice of $a$ yields marginal posteriors that are as
close as possible to the reference marginal posteriors $\Be(\theta
_{i}\g x_{i}+\h, n-x_{i}+\h)$, arising
when $\theta_i$ is the parameter of interest. Roughly, the recommended
choice is $a=1/m$, resulting
in the overall prior $\pi^o(\theta_1, \ldots, \theta_m) \propto
\prod
_{i=1}^m \theta_i^{(1 -m)/m}$.
Note that this distribution adds only $1/m = 0.001$ to each cell in
the earlier example, so that
\[
\E[\theta_i \g{\bm{x}}] = \frac{x_i+1/m}{\sum_{i=1}^{m}
(x_i+1/m)} =
\frac{x_i+1/m}{n +1}= \frac{x_i+0.001}{4} \,.
\]
Thus $\E[\theta_{240} \g{\bm{x}}] \approx0.5$, $\E[\theta_{876}
\g
{\bm{x}}] \approx0.25$, and $\E[\theta_{i} \g{\bm{x}}] \approx
\frac
{1}{4000}$ otherwise, all sensible results.
\end{exa}

\subsubsection{Hierarchical approach}
Utilize hierarchical modeling to transfer the reference prior problem
to a `higher level' of the model (following the advice of I.~J.~Good).
In this approach one
\begin{enumerate}
\item Chooses a class of {\em proper} priors $\pi(\bm{\theta}\g a)$
reflecting the desired structure of the
problem.
\item Forms the marginal likelihood
$p(\bm{x}\g a) = \int p(\bm{x}\g a)\pi(\bm{\theta}\g a) \ d\bm
{\theta}$.
\item Finds the reference prior, $\pi^R(a)$, for $a$ in this marginal model.
\end{enumerate}
Thus the overall prior becomes
\[
\pi^o(\bm{\theta}) = \int\pi(\bm{\theta}\g a)\, \pi^R(a)\,da \,,
\]
although computation is typically easier by utilizing both $\bm{\theta}$
and $a$ in the computation
rather than formally integrating out $a$.
\par

\begin{exa}
{\bf Multinomial (continued)} The Dirichlet $\Di(\bm{\theta}\g a,
\ldots
, a)$ class of priors is natural
here, reflecting the desire to treat all the $\theta_i$ similarly. We
thus need only to find the reference prior for $a$ in the marginal model,
\begin{eqnarray}
p(\bm{x}\g a) &=& \int{\left(
\begin{array}{c}
n \\
x_1 \ldots x_m \\
\end{array}
\right)}
\left( \prod_{i=1}^{m} \theta_i^{x_i}\right) \frac{\Gamma(m\,
a)}{\Gamma
(a)^m} \prod_{i=1}^{m} \theta_i^{a-1} d{\bm{\theta}} \nonumber\\
&=& { \left(
\begin{array}{c}
n \\
x_1 \ldots x_m \\
\end{array}
\right)} \frac{\Gamma(m\,a)}{\Gamma(a)^m} \frac{\prod
_{i=1}^{m}\Gamma
(x_i+a)}{\Gamma(n+m\,a)} \m.
\label{eq.good}
\end{eqnarray}
The reference prior for $\pi^R(a)$ would just be the Jeffreys-rule
prior for this marginal model; this is computed in Section~\ref
{sec.hierarchical}. The implied prior for $\bm{\theta}$ is, of course
\[
\pi(\bm{\theta}) = \int\Di(\bm{\theta}\g a)\, \pi^R(a) \,da \,.
\]
Interestingly, $\pi^R(a)$ turns out to be a proper prior, necessary
because the marginal likelihood is bounded away from zero as $a
\rightarrow\infty$.

As computations in this hierarchical setting are more complex, one
might alternatively simply choose the Type-II maximum likelihood
estimate---\emph{i.e.},\ the value of~$a$ that
\mbox{maximizes} (\ref{eq.good})---at least when $m$ is large enough so
that the empirical Bayes procedure
can be expected to be close to the full Bayes procedure. For the data
given in the earlier example (one cell having two counts, another one
count, and the rest zero counts), this marginal likelihood is
proportional to
$[a(a+1)]/[(m\,a+1)(m\,a+2)]$,
which is maximized at roughly $a = \sqrt{2}/m$. In Section~\ref
{sec.hierarchical} we will see that it is actually
considerably better to maximize the reference posterior for~$a$, namely
$\pi^R(a\g\bm{x})\propto p(\bm{x}\g a)\,\pi^R(a)$, as it can be
seen that
the marginal likelihood does not go to zero as $a \rightarrow\infty$ and
the mode may not even exist.
\end{exa}

\subsection{Outline of the paper}

Section \ref{sec.common} presents known situations in which the
reference priors for any parameter (of interest) in the model are
identical. This section is thus the beginnings of a catalogue of good
overall objective priors. Section~\ref{sec.refdistance} formalizes the
reference distance approach and applies it to two models---the
multinomial model and
the normal model where the coefficient of variation is also a
parameter of interest. In Section~\ref{sec.hierarchical} we consider
the hierarchical
prior modeling approach, applying it to three models---the multinomial
model, a hypergeometric model, and the multinormal model---and
misapplying it to the bivariate normal model. Section~\ref{discussion}
presents conclusions.

\section{Common reference prior for all parameters}
\label{sec.common}

In this section we discuss situations where the reference prior is
unique, in the
sense that it is the same no matter which of the specified model
parameters is taken
to be of interest and which of the possible
possible parameter orderings is used in the derivation. (In general, a
reference prior will
depend on the parameter ordering used in its derivation.) This unique
reference prior is typically an excellent choice for the overall prior.

\subsection{Structured diagonal Fisher information matrix}
Consider a parametric family $p(\bm{x}\g\bm{\theta})$ with unknown parameter
$\bm{\theta}=(\theta_1,\theta_2,\cdots,\theta_k)$. For any parameter
$\theta_i$, let
$\bm{\theta}_{-i}=(\theta_1,\cdots,\theta_{i-1},\theta_{i+1},
\cdots,
\theta_k)$
denote the parameters other than $\theta_i$.
The following theorem encompasses a number of important situations in which
there is a common reference prior for all parameters.

\begin{theorem} \label{th_001}
Suppose that the Fisher information matrix of $\bm{\theta}$ is of the form,
\begin{eqnarray} \label{Fisher_form_1}
\bfI(\bm{\theta})= \mbox{\textrm{diag}}
(f_1(\theta_1)g_1(\bm{\theta}_{-1}),
f_2(\theta_2)g_2(\bm{\theta}_{-2}),
\cdots,
f_k(\theta_m)g_k(\bm{\theta}_{-k})),
\end{eqnarray}
where $f_i$ is a positive function of $\theta_i$
and $g_i$ is a positive function of $\bm{\theta}_{-i},$ for
$i=1,\cdots,k.$
Then the one-at-a-time reference prior, for any chosen
parameter of interest and any ordering of the nuisance parameters in
the derivation, is given by
\begin{eqnarray} \label{prior_001}
\pi^R(\bm{\theta}) \propto
\sqrt{f_1(\theta_1) f_2(\theta_2)\cdots f_k(\theta_k)}.
\end{eqnarray}
\end{theorem}
\begin{proof}
The result follows from \citet{DatGho1996}.
\end{proof}

This prior is also what was called the {\em independent reference prior}
in \citet{SunBer1998}, and is the most natural definition of an
{\em independence Jeffreys prior} under condition (\ref{Fisher_form_1}).
Note that being the common reference prior for all of the original
parameters of interest in the model does not guarantee
that $\pi^R$ will be the reference prior for every potential parameter
of interest (see Section \ref{sec.exact})
but, for the scenarios in which an overall prior is
desired, this unique reference prior for all natural parameters is
arguably optimal.

A simple case in which (\ref{Fisher_form_1}) is satisfied is when the
density is of the form
\begin{equation}
\label{eq.simple}
p( \bm{x}\g\bm{\theta}) = \prod_{i=1}^k p_i(\bm{x}_i \g\theta
_i) \,,
\end{equation}
with $\bm{x}$ decomposable as $\bm{x}= (\bm{x}_1, \ldots, \bm
{x}_k)$. In this
case the (common to all parameters) reference
prior is simply the product of the reference priors for each of the
separate models
$p_i(\bm{x}_i \g\theta_i)$; this is also the Jeffreys-rule prior.

\subsubsection{Bivariate binomial distribution}

\citet{CroSwe1989} consider the following bivariate binomial
distribution,
whose probability density is given by
\begin{eqnarray*}
p(r,s \g\theta_1,\theta_2) ={m \choose r} \theta_1^r(1-\theta_1)^{m-r}
{r \choose s} \theta_2^s(1-\theta_2)^{r-s},
\end{eqnarray*}
where $0<\theta_1, \theta_2<1$, and $s$ and $r$ are nonnegative
integers satisfying
$0\le s \le r \le n$.
The Fisher information matrix for $(\theta_1,\theta_2)$ is given by
\begin{eqnarray*}
\bfI(\theta_1,\theta_2) = n \mbox{~diag} [\{\theta_1(1-\theta_1)\}
^{-1},~ \theta_1\{\theta_2 (1-\theta_2)\}^{-1}] \,,
\end{eqnarray*}
which is of the form (\ref{Fisher_form_1}). (Note that this
density is not of the form (\ref{eq.simple}).) Hence the reference
prior, when either $\theta_1$ or $\theta_2$ are the parameter of
interest, is
\[
\pi^R(\theta_1,\theta_2) \propto\{\theta_1(1-\theta_1) \theta
_2(1-\theta_2)\}^{-\fr{1}{2}} \,,
\]
\emph{i.e.},\ independent Beta, $\Be(\theta_i\g1/2,1/2)$
distributions for
$\theta
_1$ and $\theta_2$; this reference prior was first formally derived for
this model by \citet{PolWas1990}.
This is thus the overall recommended prior for this model.

\subsubsection{Multinomial distribution for directional data}

While we have already seen that determining an overall reference prior
for the
multinomial distribution is challenging, there is a special case of the
distribution
where doing so is possible. This happens when the cells are ordered or
directional.
For example, the cells could be grades for a class such as A, B, C, D,
and F;
outcomes from an attitude survey such as strongly agree, agree,
neutral, disagree,
and strongly disagree; or discrete survival times.
Following Example 1.1 (multinomial example), with this cell ordering,
there is a natural reparameterization of the
multinomial probabilities into the conditional probabilities
\begin{eqnarray}\label{hazard_rates_1}
\xi_j=\fr{\theta_j}{\theta_j+\cdots+\theta_{m}}, \mbox{~for~}
j=1,\cdots
,m-1 \,.
\end{eqnarray}
Here $\xi_j$ is the conditional probability of
an observation being in cell $j$ given that the observation is in cells
$j$ to $m$.
The Fisher information matrix of $(\xi_1,\cdots,\xi_{m-1})$ is
\begin{eqnarray}\label{fisher_information_1}
\bfI^*(\xi_1,\xi_2\cdots,\xi_{m-1}) = n~ \mbox{\textrm
{diag}}(\eta_1,\cdots,
\eta_{m-1}),
\end{eqnarray}
where
\begin{eqnarray*}
\delta_j &=& \frac{1}{\xi_j(1-\xi_j)} \prod_{i=1}^{j-1}(1-\xi_i),
\end{eqnarray*}
for $j=1,\cdots,m-1.$
Clearly (\ref{fisher_information_1}) is of the form (\ref{Fisher_form_1}),
from which it immediately follows that the one-at-a-time reference prior
for any of the parameters $(\xi_1,\xi_2\cdots,\xi_{m-1})$
(and any ordering of them in the derivation)
is the product of independent Beta $(1/2, 1/2)$ distributions
for the $\xi_j$ for $j=1,\ldots m-1$.
This is the same as Berger and Bernardo's (\citeyear{BerBer1992b})
reference prior for
this specific ordering of cells.

\subsubsection{A two-parameter exponential family}

Bar-Lev and Reiser (\citeyear{BarRei1982})
considered the following two-parameter
exponential family density:
\begin{eqnarray}
\label{twopar}
p(x \g\theta_1,\theta_2) &=&
a(x)\exp\{\theta_1 U_1(x) - \theta_1 G_2^{\prime}(\theta_2)
U_2(x)-\psi(\theta_1,\theta_2) \},
\end{eqnarray}
where the $U_i(\cdot)$ are to be specified, $\theta_1<0$, $\theta
_2=\expect\{ U_2(X) \g(\theta_1,\theta_2)\}$,
the $G_i(\cdot)$'s, are infinitely differentiable functions with
$G_i^{\prime\prime}>0,$
and $\psi(\theta_1,\theta_2)=-\theta_1
\{\theta_2 G_2^{\prime}(\theta_2)- G_2(\theta_2)\}+G_1(\theta_2)$.
This is a large class of distributions, which includes, for suitable
choices of $G_1$, $G_2$, $U_1$ and $U_2$, many popular statistical
models such as the normal, inverse normal, gamma, and inverse gamma.
Table 1, reproduced from \citet{Sun1994}, indicates how each
distribution arises.

\medskip
\noindent{Table 1. Special cases of Bar-Lev and Reiser's (\citeyear
{BarRei1982}) two
parameter exponential family, where $h(\theta_1)=-\theta_1+\theta_1
\log(-\theta_1) +\log(\Gamma(-\theta_1)).$ }
\begin{center}
\begin{tabular}{||c||c|c|c|c|c|c||} \hline\hline
& $G_1(\theta_1)$ & $G_2(\theta_2)$ &
$U_1(x)$ &$U_2(x)$ &$\theta_1$ & $\theta_2$ \\ \hline\hline
Normal $(\mu,\sigma)$ & $-\fr{1}{2}\log(-2\theta_1)$ & $\theta
_2^2$ &
$x^2$ & $x$ & $-1/(2\sigma^2)$ & $\mu$ \\ \hline
Inverse Gaussian & $-\fr{1}{2}\log(-2\theta_1)$ & $1/\theta_2$ &
$1/x$ & $x$ & $-\alpha/2$ & $\sqrt{\alpha/\mu}$ \\ \hline
Gamma & $h(\theta_1)$ & $-\log\theta_2$ &
$-\log x$ & $x$ & $-\alpha$ & $\mu$ \\ \hline
Inverse Gamma & $h(\theta_1)$ & $-\log\theta_2$ &
$\log x$ & $1/x$ & $-\alpha$ & $\mu$ \\ \hline\hline
\end{tabular}
\end{center}

The Fisher information matrix of $(\theta_1,\theta_2)$ based on (\ref{twopar}) is
\begin{eqnarray*} 
{\bfI}(\theta_1,\theta_2)= \left(
\begin{array}{cc}
G_1''(\theta_1) & 0\\
0& -\theta_1 G_2''(\theta_2) \\
\end{array}
\right)\,,
\end{eqnarray*}
which is of the form (\ref{Fisher_form_1}). Thus, when either $\theta
_1$ or $\theta_2$ is the parameter of interest,
the one-at-a-time reference prior (first shown in \citet{SunYe1996}) is
\begin{eqnarray} \label{twopar_ref}
\pi^R (\theta_1,\theta_2)= \sqrt{G_1''(\theta_1) G_2''(\theta_2)}.
\end{eqnarray}

For the important special case of the Inverse Gaussian density,
\begin{eqnarray}
p(x\g\alpha, \psi) = ( {\alpha}/{2 \pi x^3})^{1/2}
\exp\Bigl\{- \fr{1}{2}{\alpha x}( {1}/{x}-\psi)^2 \Bigl\}, ~x>0
\end{eqnarray}
where $\alpha>0,\psi>0$, the common reference prior (and overall
recommended prior) is
\begin{eqnarray}
\pi^R (\alpha,\psi) \propto\fr{1}{\alpha\sqrt{\psi}} \,.
\end{eqnarray}
The resulting marginal posteriors of $\alpha$ and $\psi$ can be found
in \citet{SunYe1996}.

For the important special case of the Gamma $(\alpha, \mu)$ density,
\begin{eqnarray}
p(x\g\alpha,\mu) = {\alpha^{\alpha} x^{\alpha-1} }
\exp(-{\alpha x}/{\mu})/ \{{\Gamma(\alpha) \mu^{\alpha}}\},
\end{eqnarray}
the common reference prior (and overall recommended prior) is
\begin{eqnarray}
\pi^R (\alpha,\mu) \propto\fr{\sqrt{\alpha\xi(\alpha)-1}}
{\sqrt
{\alpha} \mu},
\end{eqnarray}
where $\xi(\alpha) = (\partial^2/\partial\alpha^2) \log\{ \Gamma
(\alpha
)\}$ is
the polygamma function.
The resulting marginal posteriors of $\alpha$ and $\mu$ can be found in
\citet{SunYe1996}.

\subsubsection{A stress-strength model}
Consider the following stress-strength system, where $Y$, the
strength of the system, is subject to stress $X$. The system fails
at any moment the applied stress (or load) is greater than the strength
(or resistance). The reliability of the system is then given by
\begin{eqnarray}
\theta= P(X\le Y) \,.
\end{eqnarray}

An important instance of this situation was described in \citet
{EniGei1971}, where $X_1,\cdots,X_m$, and
$Y_1,\cdots, Y_n$ are independent random samples from exponential
distributions with unknown means
$\eta_1$ and $\eta_2$, in which case
\begin{eqnarray}
\theta=\eta_1/(\eta_1+\eta_2) \,.
\end{eqnarray}

As the data density is of the form (\ref{eq.simple}),
the (common to all parameters) reference prior is easily seen
to be $\pi^R(\eta_1,\eta_2) = 1/(\eta_1\eta_2)$, which is
also the Jeffreys prior as noted in \citet{EniGei1971}.
Our interest, however, is primarily in $\theta$.
Defining the nuisance parameter to be
$\psi=\eta_1^{(m+n)/n} \eta_2^{(m+n)/m}$,
the resulting Fisher information matrix is
\[
\bfI(\theta,\psi)=\w{diag}
\left(\frac{mn}{(m+n)\theta^2(1-\theta)^2}\m,
\frac{m^2 n^2}{(m+n)^3 \psi^2}\right),\vadjust{\eject}
\]
again of the form (\ref{Fisher_form_1}).
So the Jeffreys prior and the one-at-a-time reference prior of any
ordering for
$\theta$ and $\psi$ is
$\pi^R(\theta,\psi) = 1/\{\theta(1-\theta)\psi\}$, which can be seen
to be the transformed version of $\pi^R(\eta_1,\eta_2) = 1/(\eta
_1\eta_2)$.
So the Jeffreys prior is also the one-at-a-time reference prior for
$\theta$.
\citet{GhoSun1998} showed that this prior is the second order
matching prior
for $\theta$ when $m/n\go a >0$.

\subsection{Other scenarios with a common reference prior}
A common reference prior can exist in scenarios not covered by Theorem
\ref{th_001}.
Two such situations are considered here, the first which leads to a
fine overall
prior and the second which does not.

\subsubsection{The location-scale family}
\label{sec.LS}

Consider the location-scale family having density
\begin{eqnarray*}
p(x \g\mu, \sigma) = \frac{1}{\sigma} g\Big(\frac{x-\mu}{\sigma
} \Big),
\end{eqnarray*}
where $g$ is a specified density function and $\mu\in\reals$ and
$\sigma
>0$ are both unknown.
The Fisher information of $(\mu,\sigma)$ is
\begin{eqnarray} \label{Fisher_location_scale}
\bfI(\mu,\sigma)=\frac{1}{\sigma^2} \left(
\begin{array}{ll}
\mbox{$\int\frac{[g'(y)]^2}{g(y)} dy $} &
\mbox{$\int\{ y \frac{[g'(y)]^2}{g(y)} + g'(y) \} dy $} \\
\mbox{$\int\{ y \frac{[g'(y)]^2}{g(y)} + g'(y) \} dy $} &
\mbox{$\int\frac{[yg'(y)+g(y)]^2}{g(y)} dy $} \\
\end{array}
\right) \,.
\end{eqnarray}
Although this is not of the form (\ref{Fisher_form_1}), it is easy to
see that
the one-at-a-time reference prior for either $\mu$ or $\sigma$ is
$\pi^R(\mu,\sigma)= 1/\sigma$. This prior is also the
right-Haar prior for the location-scale group, and known to result in
Bayes procedures with optimal frequentist properties. Hence it is clearly
the recommended overall prior.

\subsubsection{Unnatural parameterizations}

A rather unnatural parameterization for the bivariate normal model
arises by defining
$\psi_1 = 1/\sigma_1, \psi_2=1/\sqrt{\sigma_2^2(1-\rho^2)},$
and $\psi_3=-\rho\sigma_2/\sigma_1.$
From \citet{BerSun2008},
the Fisher information matrix for the parameterization $(\psi_1,\psi
_2,\psi_3,\mu_1,\mu_2)$ is
\begin{eqnarray}
\bfI= \mbox{\textrm{diag}} \Big(
\frac{2}{\psi_1^2}, \frac{2}{\psi_2^2},\frac{\psi_2^2}{\psi_1^2},
\bfSigma^{-1} \Big),
\end{eqnarray}
where $\bfSigma^{-1}=
\left(
\begin{array}{ll}
\mbox{$\psi_1^2+\psi_2^2\psi_3^2$} & \mbox{$\psi_2^2\psi_3$} \\
\mbox{$\psi_2^2\psi_3$} & \mbox{$\psi_2^2$}
\end{array}
\right).$
While this is not of the form (\ref{Fisher_form_1}), direct
computation shows that the one-at-a-time reference prior for
any of these five parameters and under any ordering is
\begin{eqnarray}
\pi^R(\psi_1,\psi_2,\psi_3,\mu_1,\mu_2) = \frac{1}{\psi_1\psi_2}.
\end{eqnarray}
Unfortunately, this is equivalent to the right Haar prior,
$\pi^H(\sigma_1,\sigma_2,\rho,\mu_1,\mu_2) = \frac{1}{\sigma
_1^2(1-\rho^2)}$,
which we have argued is not a good overall prior. This suggests\vadjust
{\eject} that the
parameters used in this `common reference prior' approach need to be
natural, in some sense, to result in a good overall prior.

\section{Reference distance approach}
\label{sec.refdistance}
Recall that the goal is to identify a single overall prior $\pi
(\bm{\omega}
)$ that can be systematically used for all the parameters $\bm{\theta}=
\bm{\theta}(\bm{\omega})=\{\theta_1(\bm{\omega}),\ldots,\theta
_m(\bm{\omega}
)\}$ of
interest.
The idea of the reference distance approach is to find a $\pi(\bm
{\omega})$
whose corresponding marginal posteriors, $\{\pi(\theta_i\g\bm{x})\}
_{i=1}^m$ are close, in an average sense, to the reference posteriors
$\{\pi_i(\theta_i\g\bm{x})\}_{i=1}^m$ arising from the separate
reference priors
$\{\pi_{\theta_i}(\bm{\omega})\}_{i=1}^m$ derived under the assumption
that each of the $\theta_i$'s is of interest.
(In situations where reference priors are not unique for a parameter of
interest, we assume other considerations
have been employed to select a preferred reference prior.)
In the remainder of the paper, $\bm{\theta}$ will equal $\bm{\omega
}$, so we
will drop $\bm{\omega}$ from the notation.

We first consider the situation where the problem has an exact solution.

\subsection{Exact solution}
\label{sec.exact}

If one is able to find a single joint prior $\pi(\bm{\theta})$
whose corresponding marginal posteriors are precisely equal
to the reference posteriors for each of the $\theta_i$'s, so that, for
all $\bm{x}\in\cfX$,
\begin{equation}
\label{eq.exact-solution}
\pi(\theta_i\g\bm{x})=\pi_i(\theta_i\g\bm{x}),\quad i=1,\ldots
,m \,,
\end{equation}
then it is natural to argue that this should be an appropriate solution
to the problem.
The most important situation in which this will happen is when there is
a common reference prior
for each of the parameters, as discussed in Section \ref{sec.common}.
It is conceivable that there could be more than one
overall prior that would satisfy (\ref{eq.exact-solution}); if this
were to happen it is not
clear how to proceed.

\begin{exa}{\em Univariate normal data.}
Consider data $\bm{x}$ which consist of a random sample of normal observations,
so that \hbox{$p(\bm{x}\g\bm{\theta})=p(\bm{x}\g\mu, \sigma
)=\prod
_{i=1}^n\No
(x_i \g\mu, \sigma)$,} and suppose that one is equally interested in
$\mu$ (or any one-to-one transformation of~$\mu$) and $\sigma$ (or any
one-to-one transformation
of $\sigma$, such as the variance $\sigma^2$, or the precision
$\sigma
^{-2}$.) The common reference prior when any of these is the quantity
of interest is known to be the right Haar prior
$\pi_{\mu}(\mu, \sigma)=\pi_{\sigma}(\mu, \sigma)= \sigma^{-1}$,
and this is thus an exact solution to the overall prior problem under
the reference distance approach (as is also
clear from Section \ref{sec.LS}, since this is a location-scale family).

Interestingly, this prior also works well for making {\em joint
inferences on $(\mu, \sigma)$} in
that it can be verified that the corresponding joint credible regions
for $(\mu, \sigma)$ have appropriate coverage properties.
This does not mean, of course, that the overall prior is necessarily
good for any function of the two parameters.
For instance, if the quantity of interest is the centrality parameter
$\theta=\mu/\sigma$, the reference prior
is easily found to be
\hbox{$\pi_{\theta}(\theta,\sigma)=(1+\h\theta^2)^{-1/2}\sigma^{-1}$}
(\citealp{Ber1979}), which is not the earlier overall reference prior.
Finding a good overall prior by the reference distance situation when
this is added to the list of parameters of interest is considered
in Section \ref{sec.coeff-var}.

\end{exa}

\subsection{Reference distance solution}

When an exact solution is not possible, it is natural to consider a
family of candidate prior distributions,
$\cfF=\{\pi(\bm{\theta}\g\bfa), \bfa\in\cfA\}$, and choose, as
the overall
prior, the distribution from this class
which yields marginal posteriors that are closest, in an average sense,
to the marginal reference posteriors.

\subsubsection{Directed logarithmic divergence}

It is first necessary to decide how to measure the distance between
two distributions. We will actually use a divergence, not a distance,
namely the
directed logarithmic or Kullback-Leibler (KL) divergence \citep{KulLei1951}
given in the following definition.

\bdfn{} Let $p(\bm{\psi})$ be the probability density of a random vector
$\bm{\psi}\in\bfPsi$, and consider an approximation $p_0(\bm{\psi
})$ with the
same or larger support.
The {\em directed logarithmic divergence} of $p_0$ from $p$ is
\[
\kappa\{p_0\g p\} =\int_{\bfPsi}\; p(\bm{\psi})
\log\frac{p(\bm{\psi})}{p_0(\bm{\psi})}\,d\bm{\psi} \,,
\]
provided that the integral exists.
\label{defintrinsic}
\edfn
The non-negative directed logarithmic divergence $\kappa\{p_0\g p\} $
is the
expected log-density ratio of the true density over its approximation;
it is invariant under one-to-one transformations of the random vector
$\bm{\psi}$; and it has an operative interpretation as the amount of
information (in natural information units or \textit{nits}) which may be
expected to be required to recover $p$ from $p_0$.
It was first proposed by \citet{Ste1964} as a loss function and, in a
decision-theoretic context, it is often referred to as the {\em entropy loss}.

\subsubsection{Weighted logarithmic loss}
\label{sec.logloss}

Suppose the relative importance of the $\theta_i$ is given by a set of weights
$\{w_1,\ldots,w_m\}$, with $0<w_i<1$ and $\sum_i w_i=1$. A natural
default value for these is obviously \mbox{$w_i=1/m$}, but there are
many situations where this choice may not be appropriate; in Example
\ref{ex.multi-normal} for instance,
one might give $\theta$ considerably more weight than the means $\mu
_i$. To define the proposed criterion, we will also need to utilize
the reference prior predictives for $i=1, \ldots, m$,
\[
p_{\theta_i}(\bm{x})=\int_{\bfTheta}p(\bm{x}\g\bm{\theta}) \,
\pi
_{\theta
_i}(\bm{\theta})\,d\bm{\theta}\,.
\]

\bdfn{}
\label{def.jointloss}
The best overall prior $\pi^o(\bm{\theta})$ within the family
$\cfF=\{\pi(\bm{\theta}\g\bfa), \bfa\in\cfA\}$
is defined as that---assuming it exists and is unique---which
minimizes the {\em weighted average expected logarithmic loss}, so that
\begin{eqnarray*}
\pi^o(\bm{\theta})&=&\pi(\bm{\theta}\g\bfa^{*}), \quad\bfa
^{*}=\arg
\inf
_{\bfa\in\cfA} \,d(\bfa),\\
d(\bfa)&=&\sum_{i=1}^m w_i \int_{\cfX} \kappa\{
\pi_{\bm{\theta}_i}(\cdot\g\bm{x}, \bfa)\g\pi_{\bm{\theta
}_i}(\cdot\g
\bm{x})\}\,
p_{\theta_i}(\bm{x})\,d\bm{x},\quad\bfa\in\cfA\,.
\end{eqnarray*}
This can be rewritten, in terms of the sum of expected risks, as
\[
d(\bfa)=\sum_{i=1}^m w_i \int_{\bfTheta} \rho_i(\bfa\g\bm
{\theta}
)\,\pi
_{\theta_i}(\bm{\theta})\,d\bm{\theta},
\quad\bfa\in\cfA\,,
\]
where
\[
\rho_i(\bfa\g\bm{\theta})=
\int_{\cfX} \kappa\{\pi_{\bm{\theta}_i}(\cdot\g\bm{x}, \bfa
)\g
\pi_{\bm{\theta}_i}(\cdot\g\bm{x})\}\, p(\bm{x}\g\bm{\theta
})\,d\bm{x}
,\quad
\bm{\theta}\in\bfTheta.
\]
\edfn
Note that there is no assurance that $d(\bfa)$ will be finite if the
reference priors are improper. Indeed, in cases we have investigated
with improper reference priors, $d(\bfa)$ has failed to be finite and
hence the reference distance approach cannot be directly used.
However, as in the construction of reference priors, one can consider
an approximating sequence of proper priors
$\{\pi_{\theta_i}(\bm{\theta}\g k), k=1,2\ldots\}$ on increasing compact
sets. For each of the $\pi_{\theta_i}(\bm{\theta}\g k)$,
one can minimize the expected risk
\[
d(\bfa\g k)=\sum_{i=1}^m w_i \int_{\bfTheta} \rho_i(\bfa\g
\bm{\theta})\,\pi
_{\theta_i}(\bm{\theta}\g k)\,d\bm{\theta},
\]
obtaining $\bfa^{*}_k=\arg\inf_{\bfa\in\cfA} d(\bfa\g k)$. Then, if
$\bfa^{*}=\lim_{k\to\infty} \bfa^{*}_k$ exists, one can declare
this to
be the solution.

\subsubsection{Multinomial model}
\label{sub.multinomial}
In the multinomial model with $m$ cells and parameters $\{\theta
_1,\ldots,\theta_m\}$, with $\sum_{i=1}^m \theta_i=1$, the reference
posterior for each of the
$\theta_i$'s is $\pi_i(\theta_i\g\bm{x})=\Be(\theta_i\g x_i+\h,
n-x_i+\h)$,
while the marginal posterior distribution of $\theta_i$ resulting from
the joint prior
$\Di(\theta_1,\ldots, \theta_{m-1}\g a)$ is $\Be(\theta_i\g x_i+a,
n-x_i+(m-1)a)$.
The directed logarithmic discrepancy of the posterior $\Be(\theta_i\g
x_i+a, n-x_i+(m-1)a)$ from
the reference posterior $\Be(\theta_i\g x_i+\h, n-x_i+\h)$ is
\[
\kappa_i\{a\g\bm{x}, m, n\}=\kappa_i\{a\g x_i, m, n\}=
\kappa_{{\tiny\Be}}\{
x_i+a,n-x_i+(m-1)a
\g
x_i+\h, n-x_i+\h
\}
\]
where
\vskip-5mm
\begin{align}
& \kappa_{{\tiny\Be}}\{\alpha_0,\beta_0\g\alpha,\beta\}=
\int_0^1 \Be(\theta_i\g\alpha,\beta)
\log\Big[\frac{\Be(\theta_i\g\alpha,\beta)}{\Be(\theta_i\g
\alpha
_0,\beta_0)}\Big]\,d\theta_i \nonumber\\
&= \log\left[\frac{\Gamma(\alpha+\beta)}{\Gamma(\alpha_0+\beta
_0)}\;
\frac{\Gamma(\alpha_0)}{\Gamma(\alpha)}\;
\frac{\Gamma(\beta_0)}{\Gamma(\beta)}\right] \nonumber\\
&\hskip2mm+ (\alpha-\alpha_0)\psi(\alpha)+(\beta-\beta_0)\psi
(\beta)
-((\alpha+ \beta)-(\alpha_0+ \beta_0))\psi(\alpha+\beta),
\nonumber
\end{align}
and $\psi(\cdot)$ is the digamma function.

\begin{figure}[h!]
\includegraphics{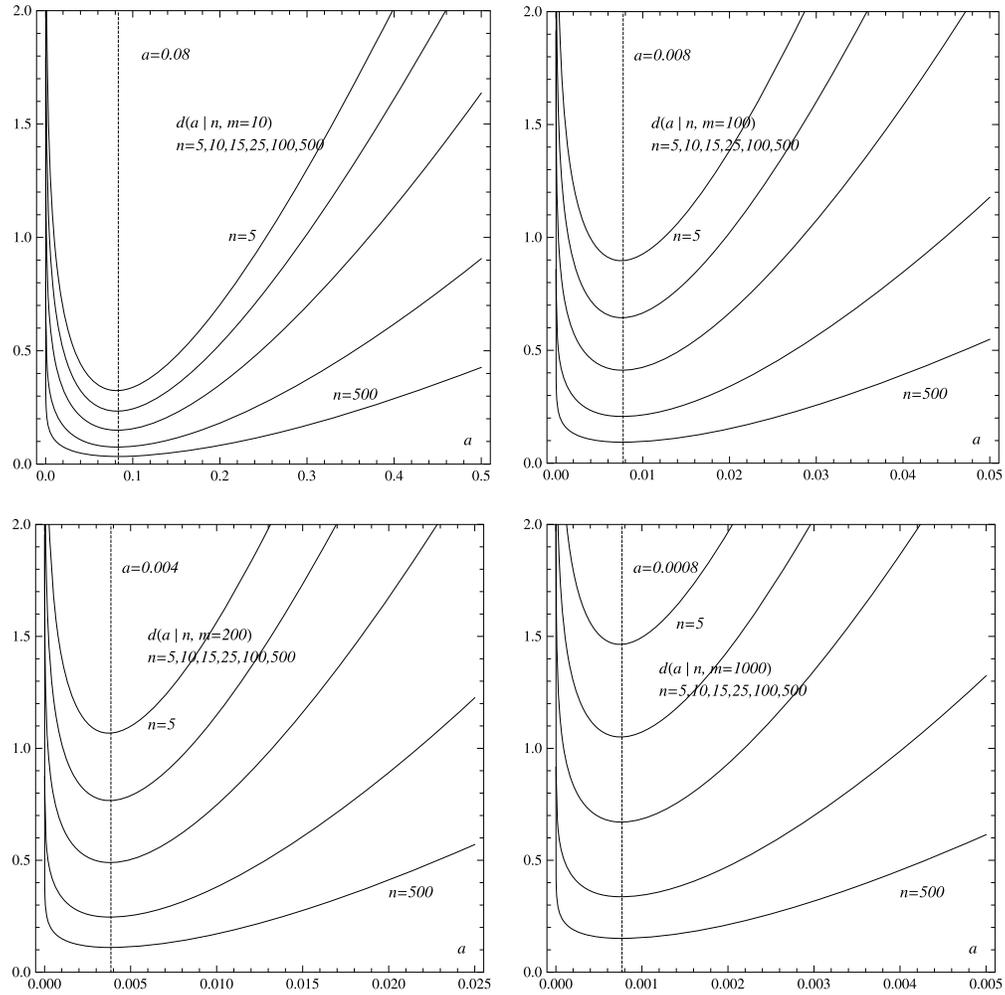}
\caption{Expected logarithmic losses, when using a Dirichlet prior
with parameter
$\{a,\ldots,a\}$, in a multinomial model with $m$ cells, for sample
sizes $n=5, 10, 25, 100$ and $500$.
The panels are for $m=10, 100,200$ and $1000$. In all cases, the
optimal value for all sample sizes is
$a^*\approx0.8/m$.}
\label{fig.multiloss}
\end{figure}

The divergence $\kappa_i\{a\g x_i, m, n\}$ between the two posteriors
of $\theta_i$ depends on the data only through $x_i$ and the sampling
distribution of $x_i$ is Binomial $\Bi(x_i\g n, \theta_i)$,
which only depends of $\theta_i$. Moreover, the marginal reference
prior for $\theta_i$ is
$\pi_{\theta_i}(\theta_i)= \Be(\theta_i\g1/2, 1/2)$ and, therefore,
the corresponding reference predictive for $x_i$
is
\[
p(x_i\g n)=\int_0^1 \Bi(x_i\g n,\theta_i)\,\Be(\theta_i\g1/2,
1/2)\,
d\theta_i
=\frac{1}{\pi}\frac{\Gamma(x_i+\h)\,\Gamma(n-x_i+\h)}{\Gamma
(x_i+1)\,
\Gamma(n-x_i+1)} \m.
\]
Hence, using Definition~\ref{def.jointloss} with uniform weights, the
average expected logarithmic loss of using a joint Dirichlet prior with
parameter $a$ with a sample of size $n$ is simply
\[
d(a\g m, n)=\sum_{x=0}^n \kappa\{a\g x, m, n\} \,p(x\g n)
\]
since, by the symmetry of the problem, the $m$ parameters $\{\theta
_1,\dots,\theta_m\}$ yield the same expected loss.

The function $d(a\g m=10, n)$ is graphed in the upper left panel of
Figure~\ref{fig.multiloss}
for several values of $n$. The expected loss decreases with $n$ and,
for any $n$, the function
$d(a\g m, n)$ is concave, with a unique minimum numerically found to be
at $a^*\approx=0.8/m=0.08$.
The approximation is rather precise. For instance, the minimum is
achieved at $0.083$ for $n=100$.

Similarly, the function $d(a\g m=1000, n)$ is graphed in the lower
right panel of
Figure~\ref{fig.multiloss}
for the same values of $n$ and with the same vertical scale, yielding
qualitatively similar results although, as one may expect, the expected
losses are now larger than those obtained with $m=10$. Once more, the function
$d(a\g m=1000, n)$ is concave, with a unique minimum numerically found
to be at $a^*\approx0.8/m=0.0008$,
with the exact value very close. For instance, for $n=100$, the minimum
is achieved at $0.00076$.

If can be concluded that, for all practical purposes when using the
reference distance approach, the best global Dirichlet prior, when one
is interested in all the parameters of a multinomial model, is that
with parameter
vector $\{1/m,\ldots,1/m\}$ (or $0.8\times\{1/m,\ldots,1/m\}$ to be
slightly more precise), yielding an approximate marginal reference
posterior for each of the $\theta_i$'s as
$\Be(\theta_i\g x_i+1/m, n-x_i+(m-1)/m)$, having mean and variance
\[
\E[\theta_i\g x_i, n]=\hat{\theta}_i=(x_i+1/m)/(n+1),\quad
\Var[\theta_i\g x_i, n]=\hat{\theta}_i(1-\hat{\theta}_i)/(n+2).
\]

\subsubsection{The normal model with coefficient of variation}
\label{sec.coeff-var}

Consider a random sample $\bfz=\{x_1,\ldots,x_n\}$ from a normal model
$\No(x\g\mu, \sigma)$, with both parameters unknown, and suppose that
one is interested in $\mu$ and $\sigma$, but also in the standardized mean
$\phi=\mu/\sigma$ (and/or any one-to-one function of them such as
$\log
\sigma$, or the coefficient of variation $\sigma/\mu$).

The joint reference prior when either $\mu$ or $\sigma$ are the
quantities of interest is
\begin{equation}
\label{eq.rhaar}
\pi_{\mu}(\mu,\sigma)=\pi_{\sigma}(\mu,\sigma)=\sigma^{-1}
\end{equation}
and this is known to lead to the Student and squared root Gamma
reference posteriors
\begin{equation*}
\label{eq.haar}
\pi^{ref}_{\mu}(\mu\g\bfz)=\St(\mu\g\mbox{$\overline
x$},s/\sqrt{n-1},n-1)\,
, \quad
\pi^{ref}_{\sigma}(\sigma\g\bfz)=\Ga^{-1/2}(\sigma\g(n-1)/2, n
s^2/2),\quad
\end{equation*}
with $n\mbox{$\overline x$}=\sum_{i=1}^nx_i$ and $n s^2=\sum_{i=1}^n
(x_i-\mbox{$\overline x$})^2$,
which are proper if $n\ge2$, and have the correct probability matching
properties.
However, the reference prior if $\phi$ is the parameter of interest is
$\pi_{\phi}(\phi,\sigma)=(2+\phi^2)^{-1/2}\sigma^{-1}$ (\citealp
{Ber1979}),
and the corresponding reference posterior distribution for $\phi$
can be shown to be
\[
\pi^{ref}_{\phi}(\phi\g\bfz)=\pi^{ref}_{\phi}(\phi\g t)\propto
(2+\phi
^2)^{-1/2}p(t\g\phi) \,,
\]
where $t=(\sum_{i=1}^n x_i)/(\sum_{i=1}^n x^2_i)^{1/2}$ has a sampling
distribution $p(t\g\phi)$ depending only on~$\phi$ (see \citealp
{StoDaw1972}). Note that all posteriors can be written in terms of the
sufficient statistics $\mbox{$\overline x$}$ and $s^2$ and the sample
size $n$, which we
will henceforth use.

A natural choice for the family of joint priors to be considered as
candidates for an overall prior is the class of {\em relatively invariant}
priors \citep{Har1964},
\[
\cfF=\{\pi(\mu,\sigma\g a)=\sigma^{-a},\, a>0\}
\]
which contains, for $a=1$, the joint reference prior~(\ref{eq.rhaar})
when either $\mu$ or
$\sigma$ are the parameters of interest, and the Jeffreys-rule prior,
for $a=2$. Since these priors are improper, a compact approximation
procedure, as described at the end of Section~\ref{sec.logloss}, is needed.
The usual compactification for location-scale parameters
considers the sets
\[
\cfC_k = \{ \mu\in(-k, k),\;\sigma\in(e^{-k}, e^k)\},\quad
k=1,2,\ldots.
\]
\par
One must therefore derive
\[
d(a\g n, k)=d_{\mu}(a\g n,k)+d_{\sigma}(a\g n,k)+d_{\phi}(a\g n,k),
\]
where each of the $d_i$'s is found by integrating the corresponding
risk with the appropriately renormalized joint reference prior. Thus,
\begin{eqnarray*}
d_{\mu}(a\g n, k)&=&\int_{\cfC_k} \left[\int_{\cfT} \kappa\{\pi
_{\mu
}(\cdot\g n, \bft,a)\g\pi^{ref}_{\mu}(\cdot\g n, \bft)\}\,
p(\bft\g n,\mu,\sigma)\,d\bft\right]\pi_{\mu}(\mu,\sigma\g
k)\,d\mu\,d\sigma,\\
d_{\sigma}(a\g n, k)&=&\int_{\cfC_k} \left[\int_{\cfT} \kappa\{
\pi
_{\sigma}(\cdot\g n, \bft,a)\g\pi^{ref}_{\sigma}(\cdot\g n, \bft
)\}\,
p(\bft\g n,\mu,\sigma)\,d\bft\right]\pi_{\sigma}(\mu,\sigma\g
k)\,d\mu
\,d\sigma,\\
d_{\phi}(a\g n, k)&=&\int_{\cfC_k} \left[\int_{\cfT} \kappa\{\pi
_{\phi
}(\cdot\g n, \bft,a)\g\pi^{ref}_{\phi}(\cdot\g n, \bft)\}\,
p(\bft\g n,\mu,\sigma)\,d\bft\right]\pi_{\phi}(\mu,\sigma\g
k)\,d\mu\,d\sigma,
\end{eqnarray*}
where $\bft=(\mbox{$\overline x$},s)$, and the $\pi_{i}(\mu,\sigma
\g k)$'s are the
joint {\em proper} prior reference densities of each of the parameter
functions obtained by truncation and renormalization in the $\cfC_k$'s.

It is found that the risk associated to $\mu$ (the expected KL
divergence of $\pi_{\mu}(\cdot\g n, \bft,a)$ from
$\pi^{ref}_{\mu}(\cdot\g n, \bft)$ under sampling) does {\em not }
depend on the parameters, so integration with the joint prior is not
required, and one obtains
\[
d_{\mu}(a\g n)=\log\left[\frac{\Gamma[n/2]\,\Gamma
[(a+n)/2-1]}{\Gamma
[(n-1)/2]\,\Gamma[(a+n-1)/2]}\right]
-\frac{a-1}{2}\left(\psi[\frac{n-1}{2}]-\psi[\frac{n}{2}]\right),
\]
where $\psi[\cdot]$ is the digamma function. This is a concave function
with a unique minimum $d_1(1\g n)=0$ at $a=1$, as one would expect from
the fact that the target family $\cfF$ contains the reference prior for
$\mu$ when $a=1$.
The function $d_{\mu}(a\g n=10)$ is the lower dotted line in
Figure~\ref
{fig.risks}.
Similarly, the risk associated to $\sigma$ does not depend either of
the parameters, and one obtains
\[
d_{\sigma}(a\g n,k)=d_{\sigma}(a\g n)=\log\left[\frac{\Gamma
[(a+n)/2-1]}{\Gamma[(n-1)/2]}\right]
-\frac{a-1}{2}\,\psi[\frac{n-1}{2}],
\]
another concave function with a unique minimum $d_2(1\g n)=0$, at $a=1$.
The function $d_{\sigma}(a\g n=10)$ is the upper dotted line in
Figure~\ref{fig.risks}.

\begin{figure}[t]
\includegraphics{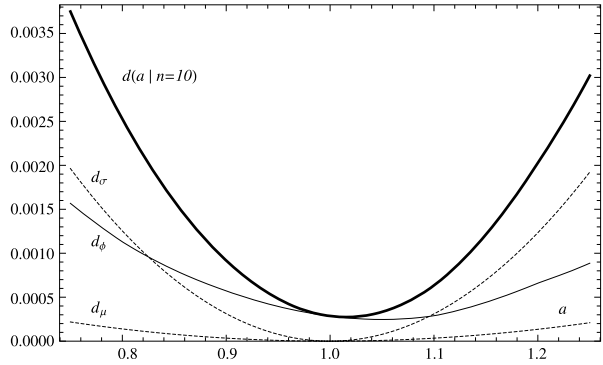}
\caption{Expected average intrinsic logarithmic losses $d(a\g n,k)$
associated with the use of the joint prior
\mbox{$\pi(\mu,\sigma\g a)=\sigma^{-a}$}
rather than the corresponding reference priors when $n=10$ and
$k=3$.}\label{fig.risks}
\end{figure}

\begin{figure}[h!]
\includegraphics{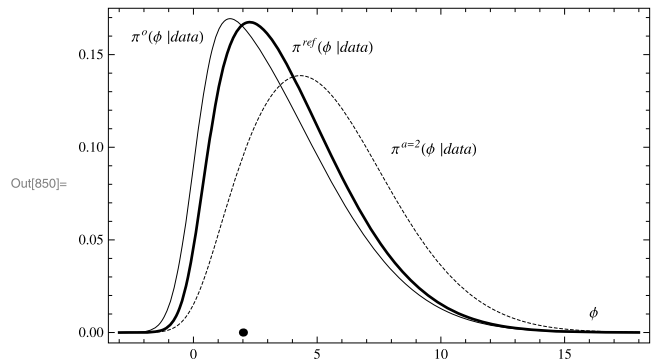}
\caption{Reference posterior (solid) and marginal overall posterior
(black) for $\phi$ given a minimal random sample of size $n=2$.
The dotted line is the marginal posterior for the prior with $a=2$,
which is the Jeffreys-rule prior.}
\label{fig.posts}
\end{figure}

The risk associated with $\phi$ cannot be analytically obtained and is
numerically computed, using one-dimensional numerical integration over
$\phi$ to obtain the KL divergence, and Monte Carlo sampling to obtain
its expected value with the truncated and renormalized reference prior
$\pi_{\phi}(\mu, \sigma\g k)$.
The function $d_{\phi}(a\g n=10,k=3)$ is represented by the black line
in Figure~\ref{fig.risks}.
It may be appreciated that, of the three components\vadjust{\eject}
of the expected
loss, the contribution corresponding to $\phi$ is the largest, and that
corresponding to $\mu$ is the smallest, in the neighborhood of the
optimal choice of $a$.
The sum of the three is the expected loss to be minimized, $d(a\g n,k)$.
The function $d(a\g n=10,k=3)$ is represented by the solid line in
Figure~\ref{fig.risks}, and has a minimum at $a^{*}_3=1.016$.
The sequence of numerically computed optimum values is $\{a^{*}_k\}=\{
1.139, 1.077, 1.016,\ldots\}$ quickly converging to some value
$a^{*}$ larger than $1$ and smaller than $1.016$, so that,
pragmatically, the overall objective prior may be taken to be the usual
objective prior for the normal model,
\[
\pi^{o}(\mu, \sigma)=\sigma^{-1}.
\]

It is of interest to study the difference in use of this overall prior
when compared with the \mbox{reference} prior for $\phi=\mu/\sigma$.
The difference is greater for smaller samples, and the minimum sample
size here is $n=2$. A random sample of two observations from $\No(x\g
1,\h)$ (so that the true value of the standardized mean is $\phi=2$)
was simulated yielding $\{x_1,x_2\}=\{0.959, 1.341\}$. The
corresponding reference posterior for $\phi$ is the solid line in
Figure~\ref{fig.posts}. The posterior that corresponds to the
recommended overall prior $a=1$ is the black line in the figure. For
comparison, the posterior corresponding to the prior with $a=2$, which
is Jeffreys-rule prior, is also given, as the dotted line. Thus, even
with a minimum sample size, the overall prior yields a marginal
posterior for $\phi$ which is quite close to that for the reference posterior.
(This was true for essentially all samples of size $n=2$ that we tried.)
For sample sizes beyond $n=4$ the differences are visually inappreciable.

\section{Hierarchical approach with hyperpriors}
\label{sec.hierarchical}

If a natural family of proper priors $\pi(\bm{\theta}\g a)$, indexed
by a
single parameter $a$, can be identified for a given problem, one can
compute the marginal likelihood $p(\bm{x}\g a)$ (necessarily a proper
density), and find the reference prior $\pi^R(a)$ for $a$ for this
marginal likelihood. This hierarchical prior specification is clearly
equivalent to use of
\[
\pi^o(\bm{\theta}) = \int\pi(\bm{\theta}\g a) \,\pi^R(a) \,da
\]
as the overall prior in the original problem.

\subsection{Multinomial problem}

\subsubsection{The hierarchical prior}
For the multinomial problem with the $\Di(\bm{\theta}\g a, \ldots,
a)$ prior,
the marginal density of any of the~$x_i$'s is
\[
p(x_i \g a, m, n) = \comb{n}{x_i} \frac{\Gamma(x_i+a)\,\Gamma(n-x_i+
(m-1)a)\,\Gamma(m\,a)}
{\Gamma(a)\,\Gamma((m-1)a)\,\Gamma(n+m\,a)} \,,
\]
following immediately from the fact that, marginally,
\[
p(x_i \g\theta_i )=\Bi(x_i \g n,\theta_i)\quad
\pi(\theta_i \g a)=\Be(\theta_i\g a,(m-1)a).
\]
Then
$\pi^R(a)$, the reference (Jeffreys) prior for the integrated model
$p(\bm{x}\g a)$ in~(\ref{eq.good}), is given in the following proposition:

\begin{teo}{}
\label{teo.hierarchical-reference-multinomial}
\begin{equation}
\label{eq.hierrefpr}
\pi^R(a\g m, n) \propto\left[\sum_{j=0}^{n-1}\left( \frac{Q(j \g
a,m,n)}{(a+j)^2} - \frac{m}{(m\,a+j)^2}\right) \right]^{1/2} \,,
\end{equation}
where $Q(\cdot\g a,m,n )$ is the right tail of the distribution of
$p(x \g a, m, n)$, namely
\[
Q(j \g a, m, n)= \sum_{l=j+1}^n p(l \g a, m, n), \quad j=0,\ldots,n-1.
\]
\end{teo}
\begin{proof}
Computation yields that
\begin{equation}
\label{eq.proof}
\E\left[ - \frac{d^2}{da^2} \log p(\bm{x}\g a) \right]= -\sum
_{j=0}^{n-1}\frac{m^2}{(m\,a+j)^2} + E\left[\sum_{i=1}^m \sum
_{j=0}^{x_i-1}\frac{1}{(a+j)^2}\right] \,,
\end{equation}
where $\sum_{j=0}^{-1} \equiv0$. Since the $x_i$ are exchangeable,
this equals
\[
-\sum_{j=0}^{n-1}\frac{m^2}{(m\,a+j)^2} + m E^{X_1}\left[\sum
_{j=0}^{X_1-1}\frac{1}{(a+j)^2}\right] \,,
\]
and the result follows by rearranging terms.
\end{proof}

\begin{teo}{}
\label{teo.hierarchical-proper}
$\pi^R(a)$ is a proper prior.
\end{teo}
\begin{proof}
The prior is clearly continuous in $a$, so we only need show that it is
integrable
at 0 and at~$\infty$. Consider first the situation as $a \rightarrow
\infty$. Then
\begin{eqnarray*}
p(0 \g a, m, n) &=& \frac{\Gamma(a)\Gamma(n+ [m-1]a)\Gamma(m\,a)}
{\Gamma(a)\Gamma([m-1]a)\Gamma(n+m\,a)} \\
&=& \frac{(m-1)a [(m-1)a+1] \cdots[(m-1)a+n-1]}{m\,a (m\,a+1) \cdots
(m\,a+n-1)} \\
&=& \frac{(m-1)}{m}(1-c_n a +O(a^2)) \,,
\end{eqnarray*}
where $c_n = 1 +1/2 + \cdots+1/(n-1)$. Thus the first term of the sum
in (\ref{eq.hierrefpr}) is
\[
\frac{1-p(0 \g a, m, n)}{a^2} - \frac{1}{m\,a^2} = \frac
{(m-1)c_n}{m\,
a} +O(1) \,.
\]
All of the other terms of the sum in (\ref{eq.hierrefpr}) are clearly
$O(1)$, so that
\[
\pi^R(a) = \frac{\sqrt{(m-1)c_n/m}}{\sqrt{a}} +O(\sqrt{a}) \,,
\]
as $a \rightarrow0$, which is integrable at zero (although unbounded).

To study propriety as $a \rightarrow\infty$, a laborious application
of Stirling's approximation yields
\[
p(x_1 \g a, m, n) = \mbox{Bi}(x_1 \g n, 1/m)(1+O(a^{-1})) \,,
\]
as $a \rightarrow\infty$. Thus
\begin{eqnarray*} \pi^R(a, m, n) &=& \left[\sum_{j=0}^{n-1} \left
(\frac
{\sum_{l=j+1}^{n}\mbox{Bi}(l \g n, 1/m)}{a^2} -\frac{1}{m\,
a^2}\right)
+O(a^{-3})\right]^{1/2} \\
&=& \left[ \left(\frac{\sum_{l=1}^{n} l\mbox{Bi}(l \g n, 1/m)}{a^2}
-\frac{n}{m\,a^2}\right) +O(a^{-3})\right]^{1/2}
= O(a^{-3/2}) \,,
\end{eqnarray*}
which is integrable at infinity, completing the proof.
\end{proof}
As suggested by the proof above, the reference prior
$\pi^R(a\g m, n)$ behaves as $O(a^{-1/2})$ near $a=0$ and behaves as
$O(a^{-2})$ for large $a$ values. Using series expansions, it is found
that, for sparse tables where $m/n$ is relatively large, the reference
prior is well approximated by the proper prior
\begin{eqnarray}
\label{eq.aproxprior}
\pi^*(a\g m, n)= \frac12\, \frac{n}{m}\,a^{-1/2}\Big(a+\frac
{n}{m}\Big)^{-3/2},
\end{eqnarray}
which only depends on the ratio $m/n$, and has the behavior at the
extremes described above.
This can be restated as saying that $\phi(a)=a/(a+(n/m))$ has a Beta
distribution $\Be(\phi\g\h,1)$.
Figure~\ref{fig.priorcont} gives the exact form of $\pi^R(a\g m, n)$
for various $(m,n)$ values, and the corresponding approximation given
by~(\ref{eq.aproxprior}). The approximate reference prior $\pi^*(a\g m,
n)$ appears to be a good approximation to the actual reference prior, and
hence can be recommended for use with large sparse contingency tables.

\begin{figure}[h!]
\includegraphics{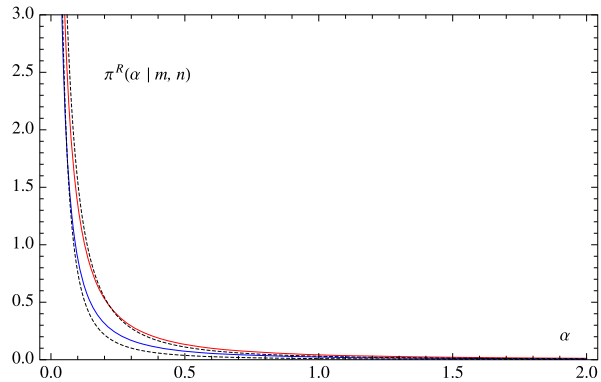}
\caption{Reference priors $\pi^R(a\g m, n)$ (solid lines) and their
approximations (dotted lines) for
$(m=150, n=10)$ (upper curve) and for $(m=500, n=10)$ (lower curve).}
\label{fig.priorcont}
\end{figure}

It is always a surprise when a reference prior turns out to be proper,
and this seems to happen when the likelihood does not go to zero at a
limit. Indeed, it is straightforward to show that
\begin{eqnarray*}
\label{eq.likebehav} p(\bm{x}\g a) = \left\{
\begin{array}{ll}
O(a^{r_0-1}), & \mbox{as $a \rightarrow0,$} \\[6pt]
{n \choose\bm{x}} \;m^{-n}, & \mbox{as $a \rightarrow\infty,$}
\end{array}
\right.
\end{eqnarray*}
where $r_0$ is the number of nonzero $x_i$.
Thus, indeed, the likelihood is constant at $\infty$, so that the prior
must be proper at
infinity for the posterior to exist.

\subsubsection{Computation with the hierarchical reference prior}
If a full Bayesian analysis is desired, the obvious MCMC sampler is as follows:
\par
{\em Step 1}. Use a Metropolis Hastings move to sample from the
marginal posterior
\newline\indent$\pi^R(a\g\bm{x}) \propto\pi^R(a)\, p(\bm{x}\g a)$.
\par
{\em Step 2}. Given $a$, sample from the usual beta posterior $\pi
(\theta\g a, \bm{x})$.
\par
This will be highly efficient if a good proposal distribution for Step
1 can be found. As it is only a one-dimensional
distribution, standard techniques should work well. Even simpler
computationally is the
use of the approximate reference prior $\pi^*(a\g m,n)$ in (\ref
{eq.aproxprior}), because
of the following result.
\begin{teo}{}
Under the approximate reference prior (\ref{eq.aproxprior}), and
provided there are at least three nonempty cells,
the marginal posterior distribution of $a$ is log-concave.
\end{teo}
\begin{proof}
It follows from (\ref{eq.proof}) that
\[
\frac{d^2}{da^2} \log[p(\bm{x}\g a) \pi^*(a\g m,n)]=
\sum_{j=0}^{n-1}\frac{m^2}{(m\,a+j)^2}
- \sum_{i=1}^m \sum_{j=0}^{x_i-1}\frac{1}{(a+j)^2}
+ \frac{1}{2a^2} + \frac{3}{2(a+n/m)^2}\m.
\]
Without loss of generality, we assume that $x_i>0,$ for $i=1,2,3$. Then
\[
\frac{d^2}{da^2} \log[p(\bm{x}\g a) p^*(a\g m,n)] <
- \sum_{i=2}^3 \sum_{j=0}^{x_i-1}\frac{1}{(a+j)^2}
+ \frac{1}{2a^2} + \frac{3}{2a^2} < 0.
\]
\vskip-8mm
\end{proof}

Thus adaptive rejection sampling \citep{GilWil1992} can be used to
sample from
the posterior of $a$.

Alternatively, one might consider the empirical Bayes solution of
fixing $a$ at its posterior mode $\widehat{a}^R$. The one caveat is
that, when $r_0=1$, it follows from (\ref{eq.likebehav}) that the
likelihood is constant at zero, while $\pi^R(a)$ is unbounded at zero;
hence the posterior mode will be $a=0$, which cannot be used. When $r_0
\geq2$, it is easy to see that $\pi^R(a) p(\bm{x}\g a)$ goes to zero
as $a \rightarrow0$, so there will be no problem.

It will typically be considerably better to utilize the posterior mode
than the maximum of $p(\bm{x}\g a)$ alone, given the fact that the
likelihood does not go to zero at $\infty$. For instance, if all
$x_i=1$, it can be shown that $p(\bm{x}\g a)$ has a likelihood
increasing in $a$, so that there is no mode. (Even when $r_0=1$, use of
the mode of $p(\bm{x}\g a)$ is not superior, in that the likelihood is
also maximized at 0 in that case.)

\subsubsection{\bf Posterior behavior as $m \rightarrow\infty$}
\label{sec.post-behavior}
Since we are contemplating the ``large sparse" contingency table
scenario, it is of considerable interest to study the behavior of the
posterior distribution as $m \rightarrow\infty$. It is easiest to
state the result in terms of the transformed variable $v=m\,a$.
Let $\pi^R_m(v \g\bm{x})$ denote the transformed reference posterior.

\begin{teo}
\begin{equation}
\label{eq.large-m-post}
\Psi(v) = \lim_{m \rightarrow\infty} \pi^R_m(v \g{\bm{x}}) =
\frac
{\Gamma(v+1)}{\Gamma(v+n)} \; v^{(r_0-\frac{3}{2})}\; \left[\sum
_{i=1}^{n-1} \frac{i}{(v+i)^2}\right]^{1/2} \,.
\end{equation}
\end{teo}
\begin{proof}
Note that
\begin{eqnarray*}
\pi^R(a \g{\bm{x}}) &\propto& m({\bm{x}} \g a)\, \pi^R( a) \\
&\propto& \frac{\Gamma( m\,a)}{\Gamma(m\,a +n)} \left[ \prod_{i=1}^m
\frac{\Gamma(a + x_i)}{\Gamma(a)}\right]\pi^R( a) \\
&\propto& \frac{\Gamma(m\,a)}{\Gamma(m\,a +n)} \left[ \prod_{i: x_i
\neq0} a (a +1) \;s (a +x_i -1) \right] \pi^R( a) \\
&\propto& \frac{\Gamma(m\,a)}{\Gamma(m\,a +n)} \left[ \prod
_{j=0}^{n-1} (a +j)^{r_j} \right] \pi^R( a)
\,,
\end{eqnarray*}
where $r_j = \{ \# x_i > j\}$.
Change of variables to $v=m\,a$ yields
\begin{eqnarray}
\pi^R_m(v \g{\bm{x}}) &\propto& \frac{\Gamma(v)}{\Gamma(v +n)}
\left[
\prod_{j=0}^{n-1} \left(\frac{v}{m} +j\right)^{r_j} \right] \pi
^R\left
(\frac{v}{m}\right) \nonumber\\
&\propto& \frac{\Gamma(v) \ v^{r_0}}{\Gamma(v +n)} \left[ C + \sum
_{i=1}^{n-r_0} K_i \left(\frac{v}{m}\right)^{i} \right] \pi^R\left
(\frac
{v}{m}\right)
\,,
\label{eq.large-m}
\end{eqnarray}
where $C= \prod_{j=2}^{n-1} j^{r_j}$ and the $K_i$ are constants.

Next we study the behavior of $\pi^R(v/m)$ for large $m$. Note first
that, in terms of $v$, the marginal density
of $x_1 =0$ is
\begin{eqnarray*}
p(0 \g v) &=& \frac{\Gamma(\frac{(m-1)}{m}v+n)}{\Gamma(\frac
{(m-1)}{m}v)} \; \frac{\Gamma(v)}{\Gamma(v +n)} \\
&=& \frac{\frac{(m-1)}{m}v[\frac{(m-1)}{m}v+1] \cdots[\frac{(m-1)}{m}v
+n-1]} {v(v+1) \cdots(v+n-1)} \\
&=& \frac{(m-1)}{m}\left(1-\frac{v}{m[v+1]}\right) \cdots\left
(1-\frac
{v}{m[v+n-1]}\right) \\
&=& \frac{(m-1)}{m}\left(1-\frac{v}{m}\sum_{i=1}^{n-1} \frac
{1}{v+i} +
O\left(\frac{v^2}{m^2 (v+1)^2}\right)\right)\,.
\end{eqnarray*}
Hence
\begin{align*}
Q(0 \g a) &= 1-p(0 \g v) \\
&= \frac{1}{m} +\frac{v(m-1)}{m^2} \sum
_{i=1}^{n-1} \frac{1}{v+i} +O\left(\frac{v^2}{m^2 (v+1)^2}\right)
=O\left(\frac{1}{m}\right) \,\, \mbox{(uniformly in $v$)}\,.
\end{align*}
It follows that all $Q(i \g a) \leq O(1/m)$, so that $\pi^R\left
(\frac
{v}{m}\right)$ is proportional to
\begin{align*}
&\left[\left(\frac{m}{v}\right)^2\left(\frac{1}{m} \,{+}\frac
{v(m-1)}{m^2} \sum_{i=1}^{n-1} \frac{1}{v+i} \right) {+} O(1)\,{+}
\sum
_{j=1}^{n-1} \frac{1}{(\frac{v}{m}+j)^2} O\left(\frac{1}{m}\right)
\,{-}
\sum_{i=0}^{n-1} \frac{m}{(v+i)^2} \right]^{1/2} \\
&\quad= \left[\sum_{i=1}^{n-1} \frac{1}{v+i}\left(\frac
{(m-1)}{v}-\frac
{m}{(v+i)}\right) + O(1) \right]^{1/2} \\
&\quad= \left[\frac{(m-1)}{v}\sum_{i=1}^{n-1} \frac{i}{(v+i)^2}+
O(1) \right
]^{1/2} \\
&\quad= \sqrt{m-1}\left[\frac{1}{v}\sum_{i=1}^{n-1} \frac{i}{(v+i)^2}+
O\left(\frac{1}{m}\right) \right]^{1/2} \,.
\end{align*}
Combining this with (\ref{eq.large-m}), noting that $v\Gamma(v)
=\Gamma
(v+1)$, and letting $m \rightarrow\infty$, yields the result.
\end{proof}

It follows, of course, that $a$ behaves like $v/m$ for large $m$, where
$v$ has the distribution in~(\ref{eq.large-m-post}).
It is very interesting that this ``large $m$" behavior of the posterior
depends on the data only through $r_0$, the number
of nonzero cell observations.

If, in addition, $n$ is moderately large (but much smaller than $m$),
we can explicitly study the behavior of the posterior mode of $a$.

\begin{teo}
Suppose $m \rightarrow\infty$, $n \rightarrow\infty$, and $n/m
\rightarrow0$. Then (\ref{eq.large-m-post}) has mode
\[
{\hat v} \approx\left\{
\begin{array}{ll}
\frac{(r_0-1.5)}{\log(1+ n/r_0)} & \mbox{if $\frac{r_0}{n}
\rightarrow
0,$} \\
c^*n & \mbox{if $\frac{r_0}{n} \rightarrow c <1,$}
\end{array}
\right.
\]
where $r_0$ is the number of nonzero $x_i$, $c^*$ is the solution to
$c^*\log(1+\frac{1}{c^*}) = c$, and $ f(n,m) \approx g(n,m)$
means $f(n,m) / g(n,m) \rightarrow1$.
The corresponding mode of the reference posterior for $a$ is $\hat{a}^R
= {\hat v}/m$.
\end{teo}

\begin{proof}
Taking the log of (\ref{eq.large-m-post}) and differentiating with
respect to $v$ results in
\[
\label{eq.deriv}
\Psi'(v) = \frac{(r_0-1.5)}{v} - \sum_{i=1}^{n-1} \frac{1}{v+i} -
\frac
{\sum_{i=1}^{n-1} \frac{i}{(v+i)^3}}{\sum_{i=1}^{n-1} \frac
{i}{(v+i)^2}} \,.
\]
Note first that, as $n$ grows, and if $v$ also grows (no faster than
$n$), then
\[
\label{eq.deriv2}
\sum_{i=1}^{n-1} \frac{1}{v+i} = \int_1^n \frac{1}{v+x} \ dx +
O\left
(\frac{1}{v+1}\right)+ O\left(\frac{1}{n}\right) = \log\left
(\frac
{v+n}{v+1}\right) + O\left(\frac{1}{v+1}\right)+ O\left(\frac
{1}{n}\right) \,.
\]
Next,
\begin{align*}
\sum_{i=1}^{n-1} \frac{i}{(v+i)^3} &= \int_1^n \frac{x}{(v+x)^3} \ dx
+ O\left(\frac{1}{(v+1)^2}\right)+ O\left(\frac{1}{n^2}\right) \\
&= \frac{1}{2} \left[\frac{(v+2)}{(v+1)^2} - \frac{(v+2n)}{(v+n)^2}
\right] {+}\, O\left(\frac{1}{(v+1)^2} + \frac{1}{n^2}\right)
{=}\, O\left(\frac{1}{v+1}\right){+}\, O\left(\frac{1}{n}\right)
,
\end{align*}
\begin{align*}
\sum_{i=1}^{n-1} \frac{i}{(v+i)^2} &= \int_1^n \frac{x}{(v+x)^2} \ dx
+ O\left(\frac{1}{(v+1)}\right)+ O\left(\frac{1}{n}\right) \\
&= \frac{v(1+n)}{(v+1)(v+n)} +\log\left(\frac{v+n}{v+1}\right) +
O\left(\frac{1}{v+1} + \frac{1}{n}\right)
\geq\log2
,
\end{align*}
again using that $v$ will not grow faster than $n$. Putting these
together we have that
\[
\Psi'(v) = \frac{(r_0-1.5)}{v} - \log\left(\frac{v+n}{v+1}\right) +
O\left(\frac{1}{v+1}\right)+ O\left(\frac{1}{n}\right) \,.
\]

{\em Case 1.} $\frac{r_0}{n} \rightarrow c$, for $0<c <1$. For this
case, write $v=c^*n/(1+\delta)$ for $\delta$ small, and note that then
\[
\Psi'(v) = \frac{c}{c^*}(1+o(1))(1+\delta) - \log\left(\frac
{(c^*+1)}{c^*}\right) + o(1) \,.
\]
Since $\frac{c}{c^*} - \log\left(\frac{(c^*+1)}{c^*}\right)=0$, it is
clear that $\delta$ can be appropriately chosen as $o(1)$ to make the
derivative zero.

\medskip
{\em Case 2.} $\frac{r_0}{n} \rightarrow0$. Now choose $v= \frac
{(r_0-1.5)}{(1+\delta)\log(1+ n/r_0)}$ and note that $\frac{v}{n}
\rightarrow0$. It follows that
\begin{align*}
&\log\left(1+\frac{n}{r_0}\right) = (\log n - \log r_0
+o(1))(1+\delta)\,\,\,\, \mbox{and} \,\,\,\,\\
&\log\left(\frac{v+n}{v+1}\right) =[\log n - \log(v+1)](1+o(1)) \,.
\end{align*}
Consider first the case $v \rightarrow\infty$. Then
\[
\log(v+1) = (1+o(1))(\log r_0 - \log\log(1+n/r_0)) = (1+o(1)) \log r_0
\,,
\]
so that
\[
\Psi'(v) = (\log n -\log r_0 +o(1)) (1+\delta) - (\log n -\log
r_0)(1+o(1)) +o(1) \,,
\]
and it is clear that $\delta$ can again be chosen $o(1)$ to make this zero.
Lastly, if $v \leq K <\infty$, then $(\log r_0)/(\log n)= o(1)$, so that
$ \Psi'(v) =(\log n)(1+o(1)) (1+\delta) - (\log n)(1+o(1)) +o(1) $, and
$\delta$ can again be chosen $o(1)$ to make this zero,
completing the proof.
\end{proof}

Table~1 gives the limiting behavior of $\hat v$ for various behaviors
of the number of nonzero cells, $r_0$. Only when $r_0=\log n$ does the
posterior mode of $a$ (i.e., $v/m$) equal $1/m$, the value selected by
the reference distance method. Of course, this is not surprising;
empirical Bayes is using a fit to the data to help select $a$ whereas
the reference distance method is pre-experimental.

\begin{table}[h!]
\begin{tabular}{c||c|c|c|c|c|}
\hline
$r_0$ & $cn$ ($0<c<1$)& $n^b$ ($0<b<1$) & $(\log n)^b$ & $\log n$ &
$O(1)$ \\
\hline
$\hat v$ & $c^*n$ & $\frac{n^b}{(1-b) \log n}$ & $(\log n)^{(b-1)}$ & 1
& $O(1/\log n)$ \\
\hline
\end{tabular}
\caption{The limiting behavior of $\hat v$ as $n \rightarrow\infty$,
for various limiting behaviors of $r_0$, the number
of non-zero cells.}\label{table.vv}
\end{table}

\subsection{ Multivariate hypergeometric model}

Let $\cfN_+$ be the set of all nonnegative integers.
Consider a multivariate hypergeometric distribution $\w{Hy}_k(\bm{r}
_k\g
n,\bfR_k,N)$
with the probability mass function
\begin{eqnarray}\label{multi_hyper_geometric}
\w{Hy}_k(\bm{r}_k\g n,\bfR_k, N)
&=&
\frac{{R_1 \choose r_1} \cdots{R_k \choose r_k}
{R_{k+1} \choose r_{k+1}}}
{{N \choose n}}\m,\quad\bm{r}_k\in\cfR_{k,n},
\end{eqnarray}
\[
\cfR_{k,n}=\{\bm{r}_k=(r_1,\cdots,r_k);\quad r_j \in\cfN_+,\quad
r_1+\cdots+r_k\le n\},
\]
where the $k$ unknown parameters $\bfR_k=(R_1,\cdots,R_k)$
are in the parameter space
$\cfR_{k,N}$.
Here and in the following, $R_{k+1}=N-(R_1+\cdots+R_k).$
Notice that the univariate hypergeometric distribution is the special
case when $k=1.$

A natural hierarchical model for the
unknown $\bfR_k$ is to assume that it is multinomial
$\w{Mu}_k(\bfR_k\g N, \bfp_k)$, with \hbox{$\bfp_k \in
\cfP_k\equiv\{\bfp_k=(p_1,\cdots,p_k)\}$}, \hbox{$0\le p_j\le1$,}
and $p_1+\cdots+p_k \le1$.
The probability mass function of $\bfR_k$ is then
\[
\w{Mu}_k(\bfR_k\g N, \bfp_k)
= \frac{N!}{\prod_{j=1}^{k+1}R_j!} \prod_{j=1}^{k+1} p_j^{R_j}.
\]
\citet{BerBerSun2012} prove that the marginal likelihood
of $\bm{r}_k$ depends only on $(n,\bfp_k)$ and it is given by
\begin{eqnarray}
p(\bm{r}_k \g\bfp_k, n, N) &=&p(\bm{r}_k \g\bfp_k, n)=
\sum_{\bfR_k\in\cfN_{k, N}}
\w{Hy}_k(\bm{r}_k\g n,\bfR_k, N) \,\w{Mu}_k \nonumber
(\bfR_k\g N, \bfp_k)
\\ &=&
\w{Mu}_k(\bm{r}_k\g n,\bfp_k), ~\bm{r}_k\in\cfR_{k,n}.
\end{eqnarray}
This reduces to the multinomial problem. Hence, the overall
(approximate) reference prior
for $(\bfR_k\g N, \bfp_k)$ would be Multinomial-Dirichlet $\Di(\bfR
_k\g
1/k,\cdots,1/k).$

\subsection{Multi-normal means}

Let $x_i$ be independent normal with mean $\mu_i$ and variance $1$,
for $i=1\cdots,m$.
We are interested in all the $\mu_i$ and in $|\bfmu|^2=\mu
_1^2+\cdots
+\mu_m^2.$

The natural hierarchical prior modeling approach is to assume that $\mu
_i \stackrel{iid}{\sim} \No(\mu_i\g0,\tau).$
Then, marginally, the $x_i$ are iid $\No(x_1\g0,\sqrt{1+\tau^2})$ and
the reference (Jeffreys) prior for~$\tau^2$ in this marginal model is
\[
\pi^R(\tau^2) \propto(1+\tau^2)^{-1}.
\]
The hierarchical prior for $\bfmu$ (and recommended overall prior) is then
\begin{eqnarray}
\label{eq.stein}
\pi^{o}(\bfmu) =
\int_0^{\infty} \frac{1}{(2\pi\tau^2)^{m/2}} \exp\left(-\frac
{|\bfmu
|^2}{2\tau^2}\right)\,
\frac{1}{1+\tau^2} \,d\tau^2 \,.
\end{eqnarray}

This prior is arguably reasonable from a marginal reference prior
perspective. For the
individual $\mu_i$, it is a shrinkage
prior known to result in Stein-like shrinkage estimates of the form
\[
\hat{\mu}_i = \left(1-\frac{r(|\bm{x}|)}{|\bm{x}|^2} \right) x_i
\,,
\]
with $r(\cdot) \approx p$ for large arguments. Such shrinkage
estimates are
often viewed as actually being superior to the reference posterior
mean, which is just
$x_i$ itself. The reference prior when $|\bf\mu|$ is the parameter
of interest is
\begin{eqnarray}
\label{eq.stein2}
\pi_{|\bfmu|}(\bfmu) \propto\frac{1}{|\bfmu|^{m-1}} \propto
\int_0^{\infty} \frac{1}{(2\pi\tau^2)^{m/2}} \exp\left(-\frac
{|\bfmu
|^2}{2\tau^2}\right)\,
\frac{1}{\tau} \,d\tau^2 \,,
\end{eqnarray}
which is similar to (\ref{eq.stein}) in that, for large values of
$|\bfmu|$, the tails
differ by only one power. Thus the hierarchical prior
appears to be quite satisfactory in terms of its marginal posterior
behavior for
any of the parameters of interest. Of course, the same could be said
for the
single reference prior in (\ref{eq.stein2}); thus here is a case where
one of the
reference priors would be fine for all parameters of interest, and
averaging among
reference priors would not work.

Computation with the reference prior in (\ref{eq.stein2}) can be done
by a simple Gibbs sampler. Computation
with the hierarchical prior in (\ref{eq.stein}) is almost as simple,
with the Gibbs
step for~$\tau^2$
being replaced by the rejection step:
\par
{\em Step 1}. Propose $\tau^2$ from the inverse gamma density
proportional to
\[
\frac{1}{(\tau^2)^{(1+m/2)}} \exp{\left( -\frac{ |\bfmu|^2}{2\tau^2}
\right)}\,,
\]
\par
{\em Step 2}. Accept the result with probability
$\tau^2/(\tau^2 + 1)$ (or else propose again).

\subsection{Bivariate normal problem}

Earlier for the bivariate normal problem, we only considered the two
right-Haar priors.
More generally, there is a continuum of right-Haar priors given as follows.
Define an orthogonal matrix by

\vskip-10mm
\begin{eqnarray}
\nonumber
\bfGamma=
\left(
\begin{array}{rr}
\cos(\beta) & -\sin(\beta) \cr
\sin(\beta) & \cos(\beta) \cr
\end{array}
\right)
\end{eqnarray}
where $-\pi/2<\beta\le\pi/2$. Then it is straightforward to see
that the right-Haar prior based on the transformed data $\bfGamma\bfX
$ is
\begin{eqnarray}
\nonumber
\pi(\mu_1, \mu_2, \sigma_1,\sigma_2,\rho\g\beta) =
\frac{\sin^2(\beta)\,\sigma_1^2+\cos^2(\beta)\,\sigma_2^2
+2\sin(\beta)\cos(\beta)\,\rho\,\sigma_1\,\sigma_2}
{\sigma_1^2\, \sigma_2^2\,(1- \rho^2)}.
\end{eqnarray}
\vadjust{\eject}

We thus have a class of priors indexed by a hyperparameter $\beta$,
and it might be tempting to try the hierarchical approach even though
the class of priors is not a class of proper priors and hence there is
no proper marginal distribution to utilize in finding the hyperprior
for $\beta$. The temptation here arises because
$\beta$ is in a compact set and it seems natural to use the (proper) uniform
distribution (being uniform over the set of rotations is natural.)
The resulting joint prior is
\[
\pi^o(\mu_1, \mu_2, \sigma_1,\sigma_2,\rho) = \frac{1}{\pi}\int
_{-\pi
/2}^{\pi/2} \pi(\mu_1, \mu_2, \sigma_1,\sigma_2,\rho\g\beta)
\,d\beta
\,,
\]
which equals the prior $\pi^A$ in (\ref{eq.S}), since
\[
\int_{-\pi/2}^{\pi/2} \sin(\beta)\cos(\beta)d\beta=0, \quad
\int_{-\pi/2}^{\pi/2} \sin^2(\beta)d\beta=
\int_{-\pi/2}^{\pi/2} \cos^2(\beta)d\beta= \mbox{constant} \,.
\]
Thus the overall prior obtained by the hierarchical approach is the
same prior as obtained by just
averaging the two reference priors. It was stated there that this prior
is inferior
as an overall prior to either reference prior individually, so the
attempt to apply the hierarchical approach to a class
of improper priors has failed.

\medskip
\noindent
{\bf Empirical hierarchical approach:} Instead of integrating out over
$\beta$, one could
find the empirical Bayes estimate $\hat\beta$ and use $\pi(\mu_1,
\mu
_2, \sigma_1,\sigma_2,\rho\g\hat\beta) $
as the overall prior. This was shown in \citet{SunBer2007} to result
in a terrible overall
prior, much
worse than either the individual reference priors, or even $\pi^A$ in
(\ref{eq.S}).

\section{Discussion}
\label{discussion}

When every parameter of a model has the same reference prior, this
prior is very natural to use as the
overall prior. A number of such scenarios were catalogued in Section
\ref{sec.common}. This common
reference prior can depend on the parameterization chosen for the
models (although it will be invariant
to coordinatewise one-to-one-transformations). Indeed, an example was
given in which a strange choice of model parameterization resulted in an
inadequate common reference prior.

The reference distance approach to developing an overall prior is
natural, and seems to work well when the reference priors themselves
are proper.
It also appears to be possible to implement the approach in the case
where the reference priors are improper, by operating on
suitable large compact sets and showing that the result is not
sensitive to the choice of compact set.
Of course, the approach is dependent on the parameterization used for
the model and on
having accepted reference priors available for all the parameters in
the model; it would have been
more satisfying if the overall prior depended only on the model itself.
The answer will also typically
depend on weights used for the individual reference priors, although
this can be viewed as a
positive in allowing more important parameters to have more influence.
The implementation
considered in this paper also utilized a class of candidate priors,
with the purpose of
finding the candidate which minimized the expected risk. The result
will thus depend on the
choice of the candidate class although, in principle, one could
consider the class of
all priors as the candidate class; the resulting minimization problem
would be formidable, however.

The hierarchical approach seems excellent (as usual), and can certainly
be recommended if one can find a natural
hierarchical structure based on a class of proper priors. Such
hierarchical structures naturally occur
in settings where parameters can be viewed as exchangeable random
variables but may not be
available otherwise. In the particular examples considered, the
overall prior obtained for the
multi-normal mean problem seems fine, and the recommended hierarchical
prior for the contingency table
situation is very interesting, and seems to have interesting
adaptations to sparsity; the same can be said for
its empirical Bayes implementation. In contrast, the attempted
application of the hierarchical and empirical Bayes idea
to the bivariate normal problem using the class of right-Haar priors
was highly unsatisfactory, even
though the hyperprior was proper.
This is a clear warning that the hierarchical or empirical Bayes
approach should be
based on an initial class of proper priors.

The failure of arithmetic prior averaging in the bivariate normal
problem was also dramatic; the initial averaging of two right-Haar
priors gave an inferior result, which was duplicated by the continuous
average over all right-Haar priors.
Curiously in this example, the geometric average of the two right-Haar
improper priors seems to be reasonable, suggesting that, if
averaging of improper priors is to be done, the geometric average
should be used.

The `common reference prior' and `reference distance' approaches will
give the same answer when a
common reference prior exists. However, the reference distance and
hierarchical approaches will rarely give the same answer because, even
if the initial class of candidate priors is the
same, the reference distance approach will fix the hyperparameter $a$,
while the hierarchical approach will assign it a reference prior; and, even
if the empirical Bayes version of the hierarchical approach is used,
the resulting estimate of $a$ can be different than that
obtained from the reference distance approach, as indicated in the
multinomial example at the end of Section \ref{sec.post-behavior}.

The `common reference prior' and hierarchical approaches will mostly
have different domains of applicability and
are the recommended approaches when they can be applied. The reference
distance approach will be of
primary utility in situations such as the coefficient of variation
example in Section \ref{sec.coeff-var}, where there is no
natural hierarchical structure to utilize nor common reference prior available.

\begin{acknowledgement}
Berger's work was supported by NSF Grants DMS-0757549-001 and DMS-1007773,
and by Grant 53-130-35-HiCi from King Abdulaziz University.
Sun's work was supported by NSF grants DMS-1007874 and SES-1260806.
The research is also supported by Chinese 111 Project B14019.
\end{acknowledgement}

\end{document}